%% file: templateArxiv_rev1.tex
\documentclass[11pt, twoside]{article}
\usepackage[margin=1in]{geometry}

\usepackage{tikz}
\usetikzlibrary{calc}

\usepackage{amsmath,amsthm}

\usepackage[utf8]{inputenc}

\usepackage{enumerate}

\usepackage{fancyhdr}
\usepackage{cite}
\usepackage{enumitem}

\usepackage{hyperref}
\hypersetup{%
    colorlinks=True, 
    citecolor=blue,
    linkcolor=black}

\usepackage[charter]{mathdesign}

\usepackage{accents}

\newcommand{\rom}[1]{\uppercase\expandafter{\romannumeral #1\relax}}

\newcommand{\blue}{\color{black}}
\newcommand{\black}{\color{black}}
\newcommand{\tow}{\rightharpoonup}
 
  \theoremstyle{definition}
  \newtheorem{theorem}{Theorem}[section]
  \newtheorem{corollary}[theorem]{Corollary}
  
  \newtheorem{lemma}[theorem]{Lemma}
  \newtheorem{definition}[theorem]{Definition}
  \newtheorem{remark}[theorem]{Remark}
   
  \newtheorem{example}[theorem]{Example}

  \newtheorem{maintheorem}{Main Theorem}

\newtheorem{thmx}{Assumption}
\newtheorem{assumption}[thmx]{Assumption}

  \newtheorem*{assumption*}{Assumption}

  \newcommand{\fu}{u}
  \newcommand{\fp}{q}
  
  \numberwithin{equation}{section}

\input{slidesymbols}

\setenumerate[0]{label=(\alph*)}

\newcommand{\eps}{{\varepsilon}}

\newcommand{\overbar}[1]{\mkern 1.5mu\overline{\mkern-1.5mu#1\mkern-1.5mu}\mkern 1.5mu}

\newcommand{\ben}{\begin{equation}}
\newcommand{\een}{\end{equation}}

\newcommand{\benn}{\begin{equation*}}
\newcommand{\eenn}{\end{equation*}}

\newcommand{\nuh}{\hat{\nu}}

\begin{document}
\title{Asymptotic analysis and topological derivative for 3D quasi-linear magnetostatics}

\author{Peter Gangl\footnote{TU Graz, Steyrergasse 30/III, 8010 Graz, Austria, gangl(at)math.tugraz.at } $\;$ and 
        Kevin Sturm\footnote{TU Wien, Wiedner Hauptstr. 8-10,
      1040 Vienna, Austria, E-Mail: kevin.sturm(at)tuwien.ac.at}}

\date{\today}

\maketitle

\begin{abstract}
 
In this paper we study the asymptotic behaviour of the quasilinear $\curl$-$\curl$ equation of 3D magnetostatics with respect to a singular perturbation of the differential operator and prove the existence of the topological derivative using a Lagrangian approach. We follow the strategy proposed in \cite{a_GanglSturm_2019a} where a systematic and concise way for the derivation of topological derivatives for quasi-linear elliptic problems in $H^1$ is introduced. In order to prove the asymptotics for the state equation we make use of an appropriate Helmholtz decomposition.

The evaluation of the topological derivative at any spatial point requires the solution of a nonlinear transmission problem. We discuss an efficient way for the numerical evaluation of the topological derivative in the whole design domain using precomputation in an offline stage. This allows us to use the topological derivative for the design optimization of an electrical machine.

\end{abstract}

%

\section{Introduction}
The main result of this paper is the computation of the topological 
derivative for the tracking-type cost function
\begin{align} \label{defJ}
	J(\Omega) &= \int_{\Omega_g} \blue | \curl u - B_d|^2 \black \; dx
\end{align}
subject to the constraint that $u\in  V(\Dsf):= \{u\in H_0(\Dsf,\curl):\; \Div(u)=0 \text{ in } \Dsf\}$ solves 
\begin{align}\label{E:weakformulation}
    \int_{\Dsf}  \Ca_\Omega(x,\curl u) \cdot \curl v \; dx = \langle F, v \rangle  \quad \mbox{for all } v \in V(\Dsf),
\end{align}
where $\Ca_\Omega:\Dsf \times \VR^3 \rightarrow \VR^3$ is a piecewise nonlinear function defined by
\ben\label{eq_AOmega}
\Ca_\Omega(x,y) := \left\{
  \begin{array}{cl}
    a_1(y) & \text{ for } x\in \Omega,\\
    a_2(y) & \text{ for } x\in \Dsf\setminus \Omega,
\end{array}\right.
\een
with two continuously differentiable functions $a_1, a_2 :\VR^3 \to \VR^3$ satisfying the following assumption:
\begin{assumption}\label{A:nonlinearity}
    There are constants $c_1, c_2, c_3>0$ such that the functions $a_i:\VR^3\to \VR^3$, $i=1,2$ are differentiable and satisfy:
    \begin{itemize}
        \item[(i)] 
            $(a_i(x)-a_i(y))\cdot (x-y) \ge  c_1 \blue |x-y|^2,\black   \quad \text{ for all } x,y\in \VR^3$,
        \item[(ii)] 
           $ \blue |a_i(x)-a_i(y)|\le  c_2|x-y| \black  \quad \text{ for all } x,y\in \VR^3$,
        \item[(iii)] 
            $  \blue |\partial a_i(x)-\partial a _i(y) |\le c_3 |x-y| \black  \quad \text{ for all } x,y\in \VR^3.$
    \end{itemize}
\end{assumption}
The right hand side $F$ is a linear and continuous functional on $V(\Dsf)$ defined by
\begin{align*}
    \langle F, v \rangle := \int_{\Omega_1} J \cdot v \; dx  + \int_{\Omega_2} M \cdot \curl v \; dx \quad  \mbox{for } v \in \blue V(\Dsf) \black,
\end{align*}
where $\Omega_1,\Omega_2\subset \Dsf$ are open sets and $J,M\in L_2(\Dsf)^3$. Properties (i) and (ii) of Assumption \ref{A:nonlinearity} imply that the operator $A_\Omega:V(\Dsf) \to \mathcal L(V(\Dsf),\VR)$ defined by $\langle A_\Omega \varphi,\psi\rangle := \int_\Dsf\Ca_\Omega(x, \curl\varphi)\cdot \curl \psi \;dx$ is Lipschitz continuous and strongly monotone for all measurable $\Omega  \subset \Dsf$. Hence the state equation \eqref{E:weakformulation} admits a unique solution by the theorem of Zarantonello; see \cite[p.504, Thm. 25.B]{b_ZE_1990a}.

Among other applications the set of equations \eqref{E:weakformulation} models a 3D electrical machine and captures nonlinear physical effects. A realistic physical model for which the above assumption are satisfied in practice will be presented in the last section.

The topological derivative has already been computed for many linear PDEs and also the literature on its numerical implementation is rich. We refer to the 
monograph \cite{NovotnySokolowski2013} for many examples and also references therein. For nonlinear PDEs the literature is far less complete and only few articles dealing with nonlinear constraints 
exist. Here we would like to mention \cite{a_AM_2006b, a_IGNAROSOSZ_2009a}, and more recently \cite{Sturm2019}, where semilinear problems were studied. 

Concerning quasi-linear problems, in which the topological perturbation enters in the main part 
of the non-linearity, even less work has been done. Here we mention \cite{a_AMBO_2017a} where the authors consider a regularised version of the $p$-Poisson equation and also \cite{AmstutzGangl2019} where the topological derivative for the quasi-linear equation of 2D magnetostatics was derived. More recently, 
in \cite{a_GanglSturm_2019a} the topological derivative for a class of quasi-linear equations under fairly general assumptions in an $H^1$ setting was presented. 

Shape optimisation for the linear Maxwell's equation has been studied in the context of time-harmonic electromagnetic waves \cite{a_HE_2012a}, magnetic impedance tomography \cite{a_HILAYO_2015a}, in electromagentic scattering \cite{a_COLE_2012a} and \cite{a_HILI_2018a}, where the last work takes a geometric viewpoint using differential forms. All these articles deal with linear problems and as far as the present authors knowledge no work has been done in the nonlinear case. In the context of optimal control in a quasi-linear $H(\curl)$ setting we mention \cite{a_YO_2013a}, where also numerical analysis is presented.

The topological sensitivity of 2D nonlinear magnetostatics, which is a simplification of Maxwell's equation in 3D, was treated in \cite{AmstutzGangl2019}. The topological sensitivity of three dimensional linear Maxwell's equations has been studied in \cite{a_MAPOSA_2005a} and is based on asymptotics derived in \cite{a_AMVOVO_2001a}. In the nonlinear context it seems no work has been done so far.

To our knowledge the asymptotics for \eqref{E:weakformulation} with respect to a singular 
perturbation of the operator is unknown. Accordingly also the topological 
derivative for the functional \eqref{defJ} and its numerical implementation are new. These are the main contribution of this paper.

The structure of the paper is as follows. In Section~\ref{sec:helmholtz} we recall 
a regular Helmholtz decomposition and prove a 
Helmholtz-type decomposition in $\VR^3$ which will be essential for the asymptotic analysis of the 
next section. In Section~\ref{sec:state_analysis} we present the asymptotic analysis 
of the state equation \eqref{E:weakformulation}. In Section~\ref{sec:topological_derivative} 
we compute the topological derivative for the cost function \eqref{defJ} using a 
Lagrangian method. In Section \ref{sec:numerics} we discuss the efficient numerical realisation 
of the obtained topological derivative. Finally, in the 
last section, we apply our results to a 3D electric machine and 
verify the pertinence of our approach in several numerical experiments.

\subsection*{Notation and definitions}

Standard $L^p$ spaces and Sobolev spaces on an open set $\Dsf\subset \VR^3$ are denoted $L_p(\Dsf)$ and $W^k_p(\Dsf)$, respectively, where $p\ge 1$ and $k\ge 1$. In case $p=2$ and $k\ge 1$  we set as usual $H^k(\Dsf):= W^k_2(\Dsf)$. Vector valued spaces are denoted $L_p(\Dsf)^3:=L_p(\Dsf,\VR^3)$ and $W^k_p(\Dsf)^3:=W^k_p(\Dsf,\VR^3)$.  Given a normed vector space $V$ we denote by $\mathcal L(V,\VR)$ the space of linear and continuous functions on $V$.  
We recall the definition of the space $H(\Dsf,\curl) = \{u\in L_2(\Dsf)^3:\curl \in L_2(\Dsf)^3\}$ and also 
\ben\label{E:curl_zero}
    H_0(\Dsf, \curl) = \left\{ u\in H(\Dsf,\curl): \int_{\Dsf}\curl u\cdot v = \int_{\Dsf}u\cdot \curl v  \quad \text{ for all } v\in H^1(\Dsf)^3  \right\}
\een
equipped with the norm $\|u\|^2_{H(\Dsf,\curl)}:= \|u\|_{L_2(\Dsf)^3}^2 + \|\curl u\|_{L_2(\Dsf)^3}^2$. It can be shown that $H_0(\Dsf,\curl) = \{ u \in L^2(\Dsf)^3 | \curl u \in L^2(\Dsf)^3 \mbox{ and } u \times n = 0 \mbox{ on } \partial \Dsf\}$. Moreover, we define the subspace
\ben
V(\Dsf) := \{u\in H_0(\Dsf,\curl):\; \Div(u)=0 \text{ on } \Dsf\}.
\een
Recall that the Friedrich's inequality $\|u\|_{L_2(\Dsf)^3} \le C\|\curl u\|_{L_2(\Dsf)^3}$ holds for all $u\in V(\Dsf)$ provided $\Dsf$ is a simply connected bounded Lipschitz domain; see \cite[Corol. 3.2]{p_SC_2018a} or \cite[Thm. 5.1]{a_BAPASC_2016a}.

We let $BL(\VR^3) :=  \{u\in H^1_{\text{loc}}(\VR^3):\; \nabla u \in L_2(\VR^3)^3\}$ 
and define the \emph{Beppo-Levi space} or \emph{homogeneous Sobolev space} as the quotient space $\dot{BL}(\VR^3) := BL(\VR^3)/\VR$, where $/\VR$ means that we quotient out the constant functions. We denote by $[u]$ the equivalence classes of $\dot{BL}(\VR^3)$. Equipped with the norm
      \ben
      \|[u]\|_{\dot{BL}(\VR^3)} := \|\nabla u\|_{L_2(\VR^3)^3}, \quad u\in [u],
      \een
      the Beppo-Levi space is a Hilbert space (see \cite{a_DELI_1955a,a_ORSU_2012a}) and $C^\infty_c(\VR^3)/\VR$ is dense in $\dot{BL}(\VR^3)$. The vector valued Beppo-Levi space $\dot{BL}(\VR^3,\VR^3)$ will be denoted by $\dot{BL}(\VR^3)^3$ and equipped with the standard norm. Whenever no confusion is possible we will not distinguish between an equivalence class $[u]$ and a representative $u$ and use the same notation. This will be clear from the context. 

      In the whole paper we equip $\VR^d$ with the Euclidean norm \blue$| \cdot |$\black and use the same notation for the corresponding matrix (operator) norm. We denote by $B_\delta(x)$ the Euclidean ball centred at $x$ with radius $\delta >0$.

\begin{remark}
    As remarked in \cite[Rem. 2.2]{a_GanglSturm_2019a}, it follows from Assumption \ref{A:nonlinearity} that the non-linearity $a_i$ satisfies:
    \begin{align} 
        \blue
        | a_i(x) | \black& \leq \blue |a_i(0)| + c_2 |x| \black, \label{rem_ai} \\
        \blue| \partial a_i(x) | \black& \leq \blue |\partial a_i(0)| + c_3 |x|\black, \label{rem_aiii} \\
        \blue|\partial a_i(x)v| \black &\le  c_2 \blue|v|\black,
    \end{align}
    for $i=1,2$ and for all $x,v \in \VR^3$.
\end{remark}

\begin{figure}
    \centering \includegraphics[width=0.5\textwidth]{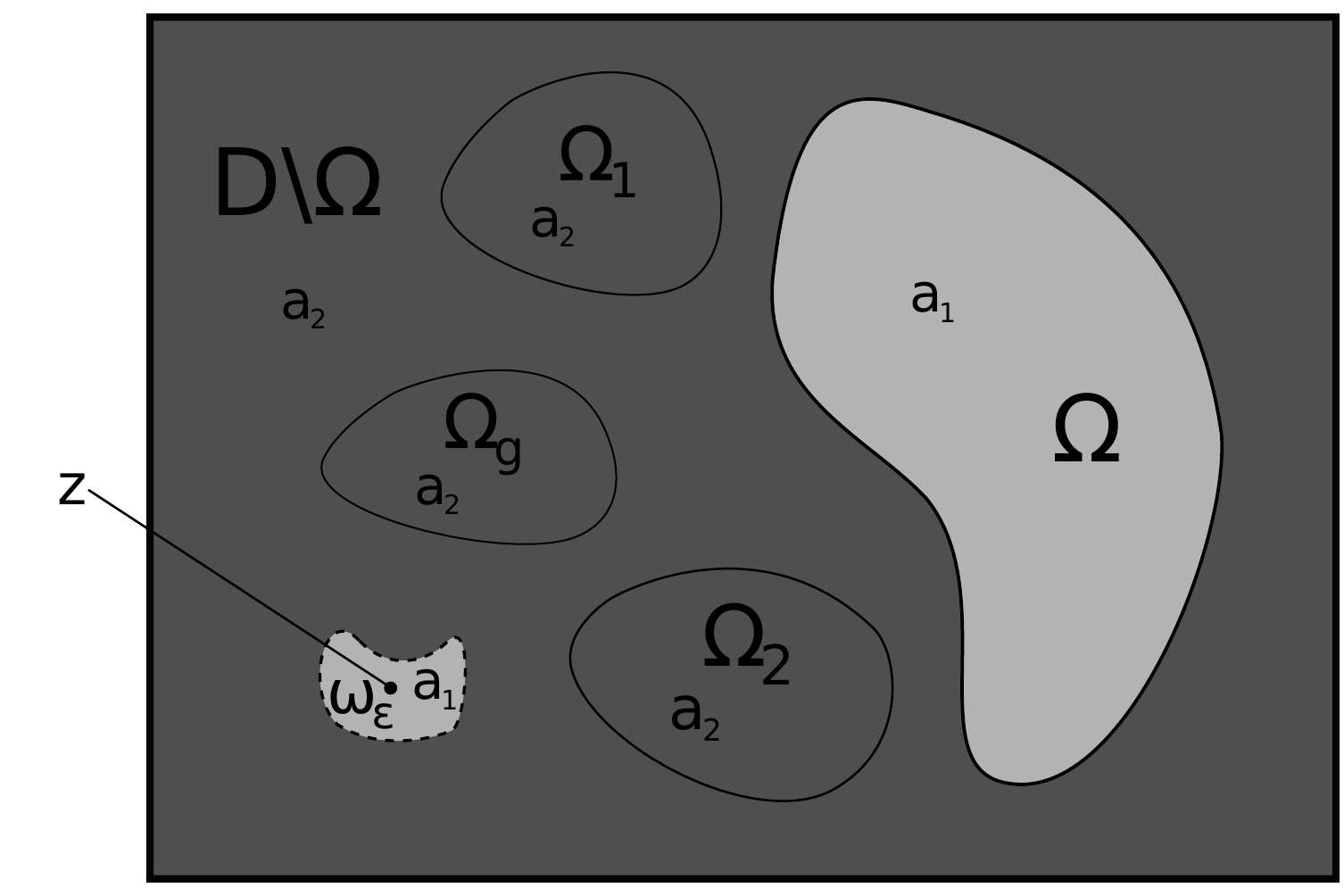}
    \caption{Setting for topological derivative: Inclusion $\omega_\eps$ of radius $\eps>0$ containing material $a_1$ around point $z \in \Dsf \setminus \overbar \Omega$ (where material $a_2$ is present).}

    \label{fig_setting}
\end{figure}

\section{Helmholtz-type decompositions in $\text{BL}(\VR^3)^3$} \label{sec:helmholtz}
In this section we develop the function space setting for the 
exterior equation that will appear in the asymptotic \blue expansion \black 
of the state equation (see Section~\ref{sec:state_analysis}). In particular we
will study a subspace of the Beppo-Levi space $\dot{BL}(\VR^3)^3$ and derive a
Helmholtz-type decomposition, which will be essential later on.
We recall the following regular Helmholtz decomposition of functions in $H_0(\Dsf,\curl)$; see, e.g., \cite[Lemma 3.4]{GiraultRaviart1986}, \cite[Thm. 29]{SchoeberlSkript} and also \cite{p_SC_2018a}. Throughout this section we assume that $\Dsf\subset \VR^3$ is a simply connected bounded Lipschitz domain. 
\begin{lemma}[Regular decomposition of $H_0(\Dsf,\curl)$] \label{L:regDecomp}
For every $u \in H_0(\Dsf,\curl)$ there exist $\varphi \in H_0^1(\Dsf)$, $u^* \in H_0^1(\Dsf)^3$ such that
	\begin{align*}
		u  = \nabla \varphi + u^*.
	\end{align*}
	Moreover, the following estimates hold:
	\begin{align*}
        \|\varphi\|_{H^1(\Dsf)} \leq C \|u\|_{H(\Dsf,\curl)} \quad \mbox{ and } \quad  \|u^*\|_{H^1(\Dsf)^3} \leq C \|\curl u\|_{L^2(\Dsf)^3}.
	\end{align*}
\end{lemma}
The following Helmholtz decomposition is standard. 
\begin{lemma}\label{lem:helmholtz_VD}
    For every $u\in H^1_0(\Dsf)^3$ we find $\varphi \in H^1_0(\Dsf)$ and $\psi \in V(\Dsf)$, such that 
    \ben\label{E:helmholtz_VD}
 u = \nabla \varphi + \psi.  
    \een
\end{lemma}
\begin{proof}
    This follows directly by solving for given $u\in H^1_0(\Dsf)^3$: find $\varphi \in H^1_0(\Dsf)$, such that
\ben
\int_\Dsf \nabla \varphi \cdot \nabla v\;dx = \int_\Dsf u\cdot \nabla v\;dx \quad \text{ for all } v\in H^1_0(\Dsf). 
\een
Then $\psi:= u-\nabla \varphi$ satisfies \eqref{E:helmholtz_VD} and $\Div(\psi)=0$. To see the boundary condition note that since $u\in H^1_0(\Dsf)^3$, we have by partial integration
\ben
 \int_{\Dsf} \curl\psi\cdot v\;dx = \int_{\Dsf} \curl u\cdot v\;dx =  \int_{\Dsf} u\cdot \curl v\;dx=  \int_{\Dsf} u\cdot \curl v\;dx - \int_{\Dsf} \nabla \varphi \cdot \curl v \, dx
\een
for all $v\in H^1(\Dsf)^3$. Here we used that the last integral vanishes, which can be seen by partial integration due to $\varphi \in H_0^1(\Dsf)$. Noting that $\psi + \nabla \varphi = u$, it follows $\psi \in V(\Dsf)$; see \eqref{E:curl_zero}. This finishes the proof.
\end{proof}

We will now introduce a subspace of the space $\dot{BL}(\VR^3)^3$. The reason why we cannot work with 
$H(\VR^3,\curl)$ directly is that we do not have control over the function $u$ itself, but only over 
its curl. In order to get around this difficulty we introduce the following function space. We also refer to \cite{a_AMVOVO_2001a} for a different approach using weighted spaces. 
\begin{definition}
    We define the space 
 \ben
 BLC(\VR^3) := \overline{\{\varphi\in C_c^\infty(\VR^3)^3:\; \Div(\varphi)=0 \}}^{|\cdot |_{H(\VR^3,\curl)}},
 \een
 where $|\varphi|_{H(\VR^3,\curl)}^2 := \int_{\VR^3} |\curl \varphi|^2\;dx$.  We set $\dot{BLC}(\VR^3) := BLC(\VR^3)/\VR$, where $/\VR$ means that we quotient out constants. 
 \end{definition} 

 We have the following result. 
 \begin{lemma}
     \begin{itemize}
         \item[(i)] We have $BLC(\VR^3)\subset BL(\VR^3)^3$ and hence $\dot{BLC}(\VR^3)\subset \dot{BL}(\VR^3)^3$.
         \item[(ii)] The space  $\dot{BLC}(\VR^3)$ becomes a Hilbert space when equipped with $|\cdot|_{H(\VR^3,\curl)}$.
         \item[(iii)] We have 
             $\dot{BLC}(\VR^3) = \{u\in \dot{BL}(\VR^3)^3:\; \Div(u)=0 \}$.
    \end{itemize}
\end{lemma}
\begin{proof}
    We start by observing that (see \cite[Rem. 1.1]{p_SC_2018a})
\ben\label{E:curl_div_grad}
\int_{\VR^3} |\Div(\varphi)|^2 + \blue |\curl(\varphi)|^2\black \;dx = \int_{\VR^3} \blue |\nabla \varphi|^2\black \;dx
\een
holds for all test functions $\varphi\in C^\infty_c(\VR^3)^3$. Therefore we have
\ben\label{E:norm_beppo_maxwell}
|\varphi|_{H(\VR^3,\curl)}^2 = \int_{\VR^3}  \blue|\curl(\varphi)|^2\black \;dx = \int_{\VR^3} \blue |\nabla \varphi|^2 \black\;dx
\een
for all test functions $\varphi\in C^\infty_c(\VR^3)^3$ satisfying $\Div(\varphi)=0$. Let $(\varphi_n)$ be a sequence in $C^\infty_c(\VR^3)^3$ with $\Div(\varphi_n)=0$ that is Cauchy with respect to 
$|\cdot|_{H(\VR^3,\curl)}$. Then in view of \eqref{E:norm_beppo_maxwell} it also converges in $\dot BL(\VR^3)^3$ and hence its limit 
belongs to $\dot BL(\VR^3)^3$, which shows the inclusion (i). Also (ii) follows at once since a closed
subspace of a Hilbert space is a Hilbert space itself.

To see (iii) we can use standard mollifier techniques; see \cite[pp. 21]{b_ZI_1989a}. Let 
$u\in L_{2,loc}(\VR^3)^3$ with $\partial u\in L_2(\VR^3)^{3\times 3}$ and $\Div(u)=0$. 
Let $\phi\in C^\infty_c(\overbar{B_1(0)})$ with $\int_{\VR^3}\phi\;dx =1$. Set $\phi_\eps(x):= \eps^{-3}\phi(x / \eps )$ and define the convolution of $u$ with $\phi_\eps$ by $u_\eps(x) := (\phi_\eps\ast u)(x) := \int_{\VR^3}\phi_\eps(x-y)u(y)\;dy$. Then $u_\eps$ is smooth, has compact support and satisfies  $\partial_{x_i} u_\eps(x) = \phi_\eps\ast (\partial_{x_i}u)(x)$ and thus $\Div(u_\eps) = \phi_\eps\ast(\Div(u)) =0$. Since $\partial_{x_i}u\in L_2(\VR^3)^3$ we conclude from \cite[Thm. 1.6.1, (iii)]{b_ZI_1989a} that $\partial_{x_i} u_\eps \to \partial_{x_i} u$ strongly in $L_2(\VR^3)^3$ as $\eps \searrow 0$. But this means that $u\in \dot{BLC}(\VR^3)$ and finishes the proof. 
\end{proof}

We now prove a Helmholtz-type decomposition in $\dot{BL}(\VR^3)^3$.
It can be seen as an analogue of Lemma~\ref{lem:helmholtz_VD} in case $\Dsf =\VR^3$. We also refer to \cite{phd_SE_2018a} and \cite{a_SISO_1996a} for Helmholtz decompositions in 
exterior domains.

Let us introduce 
\ben
BL^2(\VR^3):= \{ \varphi \in L_{2,loc}(\VR^3):\; \partial_{x_i x_j}^2\varphi \in L_2(\VR^3), \; i,j\in \{1,2,3\}\},
\een
and the associated second order Beppo-Levi space $\dot{BL}^2(\VR^3) := BL^2(\VR^3)/P$, where 
$P:= \{x\mapsto b+x\cdot a: b\in \VR, a\in \VR^3\}$ denotes
the space of linear functions in $\VR^3$.
The function  
\ben
\|\varphi\|_{BL^2}:= \|\partial^2\varphi\|_{L_2(\VR^3)^{3\times 3}}, \quad \varphi\in  \dot{BL}^2(\VR^3)
\een
is a norm on $\dot{BL}^2(\VR^3)$ and makes it a Hilbert space; see \cite[Sec. III, Thm. 2.1]{a_DELI_1955a}. 

\begin{remark}
    We note that it makes sense to say that an equivalence class $\varphi \in \dot{BL}(\VR^3)^3$ has zero divergence $\Div(\varphi)=0$, since the divergence of a constant function is zero and hence the divergence free property is independent of the representative. 
\end{remark}

\begin{lemma}\label{L:helmholtz}
    For every $u\in \dot{BL}(\VR^3)^3$ there is $\varphi \in \dot{BL}^2(\VR^3)$ and $u^*\in \dot{BL}(\VR^3)^3$ with $\Div(u^*)=0$, such that 
   \ben\label{E:helmholtz}
   u = \nabla \varphi + u^* \quad (\text{in } \dot{BL}(\VR^3)^3).
    \een
    In fact, we have the direct sum $\dot{BL}(\VR^3)^3 = \nabla(\dot{BL}^2(\VR^3))\oplus \dot{BLC}(\VR^3)^3$.
\end{lemma}
\begin{proof}
    We will use arguments from  \cite[Thm. 3.3]{a_SAVA_2004a}. Given $u\in C^\infty_c(\VR^3)^3$ we define $\varphi \in \blue C^\infty(\VR^3)^3$\black  as 
\ben
 \varphi(x) = -\frac{1}{4\pi}\int_{\VR^3} \frac{\Div u(y)}{\blue |x-y|\black }\;dy, \quad x\in \VR^3.
\een
Since $u$ is smooth and has compact support we have $\Delta \varphi = \Div u$ pointwise in $\VR^3$ (see \cite[p.21, Thm. 1]{b_EV_2010a}). The Cald\'eron-Zygmund theorem (see \cite{a_CAZY_1952a,a_CAZY_1989a} and also \blue \cite[Paragraph 9.4]{b_GITR_2001a} \black) implies that 
\ben\label{eq:bound_second_order}
\|\partial^2\varphi\|_{L_2(\VR^3)^{3\times 3}}\le C\|\Div u\|_{L_2(\VR^3)}.
\een
However, this means that $ \varphi \in \dot{BL}^2(\VR^3)$ and hence 
$u^* := \nabla \varphi - u$ satisfies $\Div(u^*)=0$ and $\nabla u^*\in L_2(\VR^3)^3$. Therefore  $u^* \in \dot{BLC}(\VR^3)$ and $u^*$ satisfies \eqref{E:helmholtz}.

Let now $u\in \dot{BL}(\VR^3)^3$ and $(u_n)\subset C_c^\infty(\VR^3)^3$ with $\partial u_n \to \partial u$ strongly in $L_2(\VR^3)^{3\times 3}$ as $n\to \infty$. The first part of the proof shows that we can split
$u_n = \nabla \varphi_n + u_n^*$ with $\varphi_n\in \dot{BL}(\VR^3)$ and $u_n^*\in \dot{BL}(\VR^3)^3$ satisfying $\Div(u_n^*)=0$. In view of \eqref{eq:bound_second_order} it follows that $(\varphi_n)$ is a Cauchy sequence in  $\dot{BL}^2(\VR^3)$ and thus converging to some $\varphi \in \dot{BL}^2(\VR^3)$.
From this it follows that also $(u_n^*)$ is a Cauchy sequence in $\dot{BL}(\VR^3)^3$ and converges to some $u^*\in \dot{BL}(\VR^3)^3$ satisfying $\Div(u^*)=0$. Now we can pass to the limit in  $\partial_{x_i} u_n = \partial_{x_i} \nabla \varphi_n + \partial_{x_i} u_n^*$, $i=1,2,3$ with respect to the $L_2(\VR^3)^3$ norm to obtain 
\ben
\partial_{x_i} (u -  \nabla \varphi - u^*)=0, \quad \text{ a.e. on }\VR^3, \quad i=1,2,3.
\een
It follows that $u- \nabla \varphi - u^*$ is constant on $\VR^3$. Therefore $u = \nabla \varphi + u^*$ in $\dot{BL}(\VR^3)^3$.  

To show that the sum is direct, we let $\tilde \varphi,\varphi \in \dot{BL}^2(\VR^3)$ and $\tilde u^*,u^*\in \dot{BLC}(\VR^3)^3$, such that $u = \nabla \varphi + u^* = \nabla \tilde \varphi + \tilde u^*$. Set $\hat \varphi := \tilde \varphi - \varphi$ and $\hat u^*:= \tilde u^* - u^*$. We have 
$\nabla \hat \varphi = -\hat u^*$ and thus since $\hat u^*$ is divergence free,  $\Delta \hat \varphi=0$, that is, $\hat\varphi$ is harmonic. By Weyl's lemma $\hat \varphi$ is smooth. Since $\hat \varphi$ is harmonic $v:= \partial^2_{x_ix_j} \hat\varphi$ is harmonic, too and hence enjoys the mean value property (see \cite[Thm. 2, p.25]{b_EV_2010a}):
\ben
v(x_0) = \frac{1}{|B_r(x_0)|}\int_{B_r(x_0)} v\;dx, \quad r>0, \; x_0\in \VR^3.
\een
Fix $x_0\in \VR^3$. Then we obtain from H\"older's inequality
\ben
|v(x_0)| \le  \frac{1}{|B_r(x_0)|}\int_{B_r(x_0)} |v|\;dx \le \frac{1}{|B_r(x_0)|^{1/2}} \|v\|_{L_2(B_r(x_0))} \le Cr^{-\frac32}\|\partial^2\hat \varphi\|_{L_2(\VR^3)^{3\times 3}}.
\een
Passing to the limit $r\to \infty$ we see that $v(x_0)=0$ and since $x_0$ was arbitrary we have $v=\partial^2_{x_ix_j} \hat\varphi = 0$ on $\VR^3$. Hence $\hat \varphi(x)=a\cdot x+b$ for some $a\in \VR^3,b\in \VR$ and thus the corresponding equivalence class $\hat \varphi =0$ in $\dot{BL}^2(\VR^3)^3$ or equivalently $\varphi=\tilde\varphi$ as elements in $\dot{BL}^2(\VR^3)^3$. In view of $a = \nabla \hat \varphi = -\hat u^*$ it follows that $\hat u^*=0$ in $\dot{BL}(\VR^3)^3$ or equivalently $u^*=\tilde u^*$. This shows that we have a direct sum. 

\end{proof}

The following example illustrates the usefulness of the function space $\dot{BLC}(\VR^3)$ and the Helmholtz-type decomposition.
\begin{example}
Let $\zeta \in \VR^3$ be a vector and let $\omega\subset \VR^3$ be an open 
and bounded set. Consider the problem: find $K\in \dot{BLC}(\VR^3)$ such that
\ben\label{EX:example}
\int_{\VR^3} \beta_\omega \curl K\cdot \curl v\;dx = \int_\omega \zeta \cdot \curl v\;dx \quad \text{ for all } v\in \dot{BLC}(\VR^3),
\een
where $\beta_\omega:= \beta_1 \chi_\omega + \beta_2 \chi_{\VR^3\setminus \omega}$ with $\beta_1,\beta_2>0$. This system appears in the derivation of the topological derivative for Maxwell's equation in the linear case; see \cite{a_MAPOSA_2005aa} on page 553. Thanks to the theorem of Lax-Milgram there exists a unique solution $K$ of \eqref{EX:example} in $\dot{BLC}(\VR^3)$. Moreover, for given $v\in \dot{BL}(\VR^3)^3$ we find according to Lemma~\ref{L:helmholtz} the decomposition $v = \nabla \varphi + v^*$ with $\varphi\in \dot{BL}^2(\VR^3)$ and $v^* \in \dot{BLC}(\VR^3)$. Therefore plugging $v^* = v - \nabla \varphi$ as test function in \eqref{EX:example} and using $\curl(\nabla \varphi)=0$, we obtain 
\ben\label{EX:example2}
\int_{\VR^3} \beta_\omega \curl K\cdot \curl v\;dx = \int_\omega \zeta \cdot \curl v\;dx \quad \text{ for all } v\in BL(\VR^3)^3.
\een
\end{example}


\section{Asymptotics of the state equation}\label{sec:state_analysis}
\subsection{Main result for direct state}
In this section we study the behaviour of $u_\eps -u_0$, where, \blue for $\eps>0$ \black, $u_\eps\in V(\Dsf)$ is the solution to 
\begin{align}\label{eq:state_per}
    \int_{\Dsf}  \Ca_\eps(x,\curl u_\eps ) \cdot \curl v \; dx = \langle F, v \rangle  \quad \mbox{for all } v \in V(\Dsf),
\end{align}
with $\Ca_\eps := \Ca_{\Omega_\eps}$ and $\Omega_\eps(z)  := \Omega \cup \omega_\eps(z)$ \blue and $u_0$ is the solution to \eqref{E:weakformulation}\black. \blue Here, $\Omega \Subset \Dsf$ is open and the scaled inclusion $\omega_\eps(z):= z + \eps \omega$ is defined by an open and bounded set $\omega\subset \VR^d$ satisfying $0\in \omega$ and the center of the inclusion $z \in \Omega \cup \Dsf \setminus \overbar \Omega$. For simplicity and without loss of generality, we will assume $z:=0 \in \Dsf\setminus \overbar \Omega$ throughout this paper.
\black


Using Lemma~\ref{L:regDecomp} we find the regular decomposition
\ben\label{E:deco_ueps}
u_\eps = \nabla \phi_\eps + u_\eps^*, \quad \phi_\eps\in H^1_0(\Dsf), \; u_\eps^* \in H^1_0(\Dsf)^3.
\een

\blue
\begin{definition}
The variation of $u_\eps$ is defined by 
\ben\label{E:var_Keps}
K_\eps  := \left(\frac{u_\eps - u_0}{\eps}\right) \circ T_\eps \in V(\eps^{-1}\Dsf), \quad \eps >0,
\een
and the variation of $u_\eps^*$ is defined by  
\ben\label{E:var_Keps_star}
K_\eps^*  := \left(\frac{u_\eps^* - u_0^*}{\eps}\right) \circ T_\eps \in H^1_0(\eps^{-1}\Dsf)^3, \quad \eps >0,
\een
where $T_\eps(x) := \eps x$ for $x\in \VR^3$. By extending $u_\eps^*$ by zero outside of $\Dsf$ we can view $K_\eps^*$ as a function in $\dot{BL}(\VR^3)^3$. 
\end{definition}
\black 

Now we can state our first main theorem.
\begin{maintheorem}\label{maintheorem_1}
    Assume  that $\curl u_0\in C^{\alpha}(\overline{B_\delta(z)})^3$ for some $\delta>0$ and $0<\alpha<1$. Then we have
    \begin{itemize}
        \item[(i)]  There exists a unique $K\in \dot{BLC}(\VR^3)$, such that
  \ben\label{E:limit_equation_W}
\begin{split} 
  \int_{\VR^3} & (\Ca_\omega (x, \curl K + U_0 )  -  \Ca_\omega (x, U_0  ))  \cdot \curl \varphi \; dx =  \\
								   &  - \int_\omega \big( a_1( U_0  ) - a_2(U_0)) \big) \cdot \curl \varphi \; dx
\end{split}
\een
for all $ \varphi \in BLC(\VR^3)$. Here $ U_0:= \curl(u_0)(z)$ and 
$\Ca_\omega(x,y) := a_1(y)\chi_{\omega}(x) +  a_2(y)\chi_{\VR^3\setminus \omega}(x)$.
\item[(ii)] 
    The family $(K_\eps^*)$ defined in \eqref{E:var_Keps_star}, satisfies
    \begin{align}
    \curl(K_\eps^*) & \rightarrow \curl(K) && \quad  \text{ strongly in } L_2(\VR^3)^3 \quad \text{ as }\eps \searrow 0. 
\end{align}
\end{itemize}
\end{maintheorem}
\begin{proof} 
Proof of (i): Thanks to Assumption~\ref{A:nonlinearity} the operator $B_\omega: \dot{BLC}(\VR^3) \to \dot{BLC}(\VR^3)^*$ defined by $\langle B_\omega\varphi,\psi\rangle := \int_{\VR^3} (\Ca_\omega(x,  \curl \varphi + U_0  ) - \Ca_\omega(x, U_0 ))\cdot \curl \psi \; dx$ is strongly monotone and Lipschitz continuous and hence the unique solvability follows by the theorem of Zarantonello; see \cite[p.504, Thm. 25.B]{b_ZE_1990a}. 

The proof of (ii) is given in the subsequent sections.
\end{proof}

Before turning our attention to the proof of (ii) let us make several remarks.
\begin{remark}
Notice that the regular decomposition \eqref{E:deco_ueps} is not necessarily unique. However, 
if we find another $\tilde \varphi_\eps\in H^1_0(\Dsf)$ and $\tilde u_\eps^* \in H^1_0(\Dsf)^3$ with $u_\eps = \nabla \tilde \varphi_\eps + \tilde u_\eps^*$, then 
$\curl(u_\eps) = \curl(u_\eps^*) = \curl(\tilde u_\eps^*)$, so $\curl(u_\eps)$ and accordingly $\curl(K_\eps^*)$ does not depend 
on the choice of decomposition in \eqref{E:deco_ueps}.  
\end{remark}

\begin{remark}
    We refer to \cite{a_AMVOVO_2001a} for a different functional setting (employed in the linear case) using weighted Sobolev spaces rather than subspaces of the Beppo-Levi space.
\end{remark}

\begin{remark}\label{rem:test_function_limit}
    Let us make an important remark. Equation \eqref{E:limit_equation_W} is actually only allowed to be tested with functions $v\in \dot{BL}(\VR^3)^3$ with $\Div(v)=0$. However, we can in fact test this equation with all functions in $\dot{BL}(\VR^3)^3$. To see this let $v\in \dot{BL}(\VR^3)^3$ be arbitrary. Thanks to Lemma~\ref{L:helmholtz} we find $\varphi\in \dot{BL}^2(\VR^3)$ and $v^*\in BL(\VR^3)^3$, such that $v = \nabla \varphi + v^*$. Since $v^* \in \dot{BLC}(\VR^3)$ we can use
$v^* = v - \nabla \varphi$ as test function in \eqref{E:limit_equation_W} and using $\curl \nabla \varphi =0$ we obtain 
  \ben\label{E:limit_equation_W2}
\begin{split}
  \int_{\VR^3} & (\Ca_\omega (x, \curl K + U_0 )  -  \Ca_\omega (x, U_0  ))  \cdot \curl v \; dx =  \\
								   &  - \int_\omega \big( a_1( U_0  ) - a_2(U_0)) \big) \cdot \curl v \; dx
\end{split}
\een
for all $ v \in BL(\VR^3)^3$. This will be used later on.
\end{remark}

\subsection{Analysis of the perturbed state equation}
We assume in the whole section that $\curl u_0\in C(\overline{B_\delta(z)})^3$ for some $\delta >0$. Moreover we assume that Assumption~\ref{A:nonlinearity}(i),(ii) are satisfied. 
 Let $u_\eps$ denote the solution to \eqref{eq:state_per}.
\begin{lemma}\label{lem:u_ueps}
	There is a constant $C>0$, such that for all small $\eps>0$, 
	\ben\label{eq:est_u_eps_D}
	\|u_\eps - u_0\|_{L_2(\Dsf)^3}+	\|\curl(u_\eps - u_0)\|_{L_2(\Dsf)^3} \le C\eps^{\blue 3/2\black}. 
	\een
\end{lemma}
\begin{proof}
Subtracting \eqref{eq:state_per} for $\eps >0$ and $\eps =0$ yields
\ben\label{E:u_eps_difference}
\begin{split}
    \int_{\Dsf} (\Ca_\eps(x,\curl u_\eps) - & \Ca_\eps(x,\curl u_0))\cdot \curl\varphi \;dx \\ 
                                            & = - \int_{\omega_\eps} (a_1(\curl u_0) - a_2(\curl u_0))\cdot \curl \varphi\;dx, 
\end{split}
\een
for all $\varphi \in V(\Dsf)$. Hence choosing $\varphi = u_\eps - u_0$ as a test function, using H\"older's inequality and the monotonicity of $\Ca_\eps$, yield
\ben
c\|\curl (u_\eps - u_0)\|_{L_2(\Dsf)^3}^2 \le C\sqrt{|\omega_\eps|}( 1 + \|\curl u_0\|_{C(\overbar B_\delta(z))^3})\|\curl (u_\eps - u_0)\|_{L_2(\Dsf)^3},
\een
where we used \eqref{rem_ai}.  Now the result follows from $|\omega_\eps|= |\omega|\eps^3$ and the Friedrich's inequality.
\end{proof}

A direct consequence of Lemma~\ref{lem:u_ueps} and Lemma~\ref{L:regDecomp} is the following. 
Recall the splitting $u_\eps = \nabla \phi_\eps + u_\eps^*$ introduced in \eqref{E:deco_ueps}.
\begin{corollary}
Under the assumptions of Lemma~\ref{lem:u_ueps},  there are constants $C_1,C_2$, such that for all small $\eps >0$ we have
\ben
\|u_\eps^* - u_0^*\|_{L_2(\Dsf)^3}+	\|\partial(u_\eps^* - u^*_0)\|_{L_2(\Dsf)^{3\times 3}} \le C_1\eps^{\blue 3/2 \black}
\een
and 
\ben
\|\phi_\eps - \phi\|_{L_2(\Dsf)}+	\|\nabla(\phi_\eps - \phi)\|_{L_2(\Dsf)^3} \le C_2\eps^{\blue 3/2\black }.
\een
\end{corollary}

The proof of Main Theorem~\ref{maintheorem_1} is split into several lemmas. 
The outline of the proof is as follows:
\blue
\begin{itemize}
    \item introduce an auxiliary function $H_\eps$ and decompose it into $H_\eps = \nabla \varphi_\eps + H_\eps^*$ 
    \item split \blue $K_\eps^* - K = K_\eps^* - H_\eps^* + H_\eps^* - K$ 
    \item show $\curl(H_\eps^* - K)\to 0$ strongly in $L_2(\VR^3)^3$
    \item show $\curl(H_\eps^* - K_\eps^*)\to 0$ strongly in $L_2(\VR^3)^3$
\end{itemize}
\black
The proof is following the main arguments of Theorem 4.3 in \cite{a_GanglSturm_2019a}. The main difference is that we cannot directly work with $K_\eps$ and $H_\eps$ but have to work with the functions $K_\eps^*$ and $H_\eps^*$ coming from the regular Helmholtz decomposition as in
Lemma~\ref{L:regDecomp}. 

\blue Let us first investigate the variation $H_\eps^* - K$. \black We start by changing variables in \eqref{E:u_eps_difference} to obtain an equation for $K_\eps^*\in H^1_0(\eps^{-1}\Dsf)^3$:
\ben\label{E:per_eps}
\begin{split}
    \int_{\VR^3} &  (\Ca_\omega(x,\curl K_\eps^* + \curl u_0(x_\eps) ) - \Ca_\omega(x, \curl u_0(x_\eps) ))\cdot \curl \varphi  \; dx \\
                &  = - \int_{\omega}(a_1 (\curl u_0 (x_\eps) ) - a_2 (\curl u_0(x_\eps) )) \cdot \curl \varphi \;dx
\end{split}
\een
for all $\varphi \in V(\eps^{-1}\Dsf)^3$. \blue Here $x_\eps := \eps x $ and $\curl u_0(x_\eps)$ denotes the $\curl $ of $u_0$ evaluated at $x_\eps$. \black 

We now introduce an approximation $H_\eps$ of $K_\eps$.
\begin{definition}
    We  define $H_\eps \in V(\eps^{-1}\Dsf)$ as the solution to 
\ben\label{E:approx_Keps}
\begin{split}
    \int_{\eps^{-1}\Dsf}  & (\Ca_\omega(x,\curl H_\eps + U_0 ) - \Ca_\omega(x, U_0 ))\cdot \curl \varphi \; dx\\
              & = - \int_{\omega}(a_1 (U_0) - a_2 (U_0)) \cdot \curl \varphi \;dx \quad \text{ for all } \varphi \in V(\eps^{-1} \Dsf).
\end{split}
\een
\end{definition}
\begin{remark}
    We can replace $V(\eps^{-1} \Dsf) $ as test space in \eqref{E:per_eps} and also in \eqref{E:approx_Keps}  by $H^1_0(\eps^{-1}\Dsf)$. Indeed in view of Lemma~\ref{lem:helmholtz_VD} we can decompose every $v\in H^1_0(\eps^{-1}\Dsf)^3$ as $v = \nabla \varphi + \psi$ with $\varphi \in H^1(\eps^{-1}\Dsf)$ and $\psi\in V(\eps^{-1}\Dsf)^3$. Hence we may test \eqref{E:approx_Keps} with $\varphi = \psi$ and using $\curl(\nabla \varphi)=0$ implies that we can test \eqref{E:approx_Keps} with all functions in $v\in H^1_0(\eps^{-1}\Dsf)^3$. Compare the $\VR^3$ analogue discussed in Remark~\ref{rem:test_function_limit}.
\end{remark}

Again we invoke Lemma~\ref{L:regDecomp} to decompose $H_\eps = \nabla \varphi_\eps + H_\eps^*$, $H_\eps^* \in H_0^1(\eps^{-1}\Dsf)^3$ and $\varphi_\eps \in H_0^1(\eps^{-1}\Dsf)$. We now introduce the projection of $K$ into the space $H^1_0(\eps^{-1}\Dsf)^3$:
    \begin{definition}\label{D:projection}
We define  $\blue \hat K_\eps^*\black \in H^1_0(\eps^{-1}\Dsf)^3$ as the minimiser of 
    \ben\label{E:hat_K_eps}
    \min_{\substack{\varphi \in H^1_0(\eps^{-1}\Dsf)^3\\ \Div\varphi =0 }} \| \curl(\varphi - K)\|_{L_2(\eps^{-1}\Dsf)^3}. 
    \een
    \end{definition}
    \blue
    The minimisation problem \eqref{E:hat_K_eps} admits indeed a unique solution. To see this, we let $\varphi_n \in H^1_0(\eps^{-1}\Dsf)^3$ be a minimising sequence, such that $\Div(\varphi_n) =0$ and
    \ben\label{eq:min_sequence}
\lim_{n\to \infty} \|\curl(\varphi_n - K)\|_{L_2(\eps^{-1}\Dsf)^3} = \inf_{\substack{\varphi \in H^1_0(\eps^{-1}\Dsf)^3\\ \Div\varphi =0 }} \| \curl(\varphi - K)\|_{L_2(\eps^{-1}\Dsf)^3}.
    \een
    Since the infimum on the right hand side is finite we conclude that there is $C>0$, such that $\|\curl \varphi_n\|_{L_2(\eps^{-1}\Dsf)^3} \le C$ for all $n$. On the other hand in view of \eqref{E:curl_div_grad} and $\Div(\varphi_n)=0$ we have
    \ben
\|\curl \varphi_n\|_{L_2(\eps^{-1}\Dsf)^3} = \|\nabla \varphi_n\|_{L_2(\eps^{-1}\Dsf)^3}.
    \een
    Therefore $(\varphi_n)$ is bounded in $H^1_0(\eps^{-1}\Dsf)$ and we find a weakly converging subsequence (denoted the same) converging to some element $\varphi\in H^1_0(\eps^{-1}\Dsf)$ satisfying $\Div(\varphi)=0$. Since also $\curl(\varphi_n) \rightharpoonup \curl(\varphi)$ weakly in $L_2(\eps^{-1}\Dsf)^3$, we conclude 
    \ben
    \|\curl(\varphi-K)\|_{L_2(\eps^{-1}\Dsf)} \le \lim_{n\to \infty} \|\curl(\varphi_n - K)\|_{L_2(\eps^{-1}\Dsf)^3},
    \een
    which together with \eqref{eq:min_sequence} shows that \eqref{E:hat_K_eps} admits a solution.  The uniqueness follows from that fact that $\varphi \mapsto \|\curl(\varphi)\|_{L_2(\eps^{-1}\Dsf)^3}^2$ is strictly convex on $\{\varphi\in H^1_0(\eps^{-1}\Dsf):\; \Div(\varphi)=0\}$.  
\black

    As for \blue $K_\eps^*$\black, we can also view \blue$H_\eps^*$ and $\hat K_\eps^*$ \black as elements of $BL(\VR^3)^3$ by extending them by $0$ outside of $\eps^{-1}\Dsf$.
    
\begin{lemma}\label{L:projection_convergence}
 It holds that
        \ben\label{E:min_project}
            \curl \hat K_\eps^* \to \curl K \quad  \mbox{ strongly in }L_2(\VR^3)^3 \text{ as }  \eps \searrow 0.
        \een       
    \end{lemma}
    \begin{proof}
        We readily check that the minimiser to \eqref{E:hat_K_eps} satisfies
        \ben\label{E:hat_K_eps2}
        \int_{\eps^{-1} \Dsf} \curl \blue \hat K_\eps^* \black \cdot \curl \varphi \;dx = \int_{\eps^{-1} \Dsf} \curl K \cdot \curl \varphi \;dx \quad \text{ for all } \varphi\in H^1_0(\eps^{-1}\Dsf)^3, \; \Div(\varphi)=0. 
        \een
        Choosing $\varphi = \blue \hat K_\eps^* \black $ and using H\"older's inequality and the fact that (see \eqref{E:curl_div_grad})
        \ben
        \|\curl v\|_{L_2(\eps^{-1}\Dsf)^3} = \|\nabla v\|_{L_2(\eps^{-1}\Dsf)^3} \quad \text{ for all } v\in H_0^1(\eps^{-1}\Dsf)^3 \text{ with } \Div(v)=0,
        \een
        we obtain 
        \ben
        \begin{split}
            \|\nabla  \blue \hat K_\eps^* \black\|_{L_2(\eps^{-1}\Dsf)^3}^2 & = \|\curl \blue \hat K_\eps^* \black \|_{L_2(\eps^{-1}\Dsf)^3}^2 \\
                                                           & \le \|\curl K\|_{L_2(\eps^{-1}\Dsf)^3}\|\curl \blue \hat K_\eps^* \black\|_{L_2(\eps^{-1}\Dsf)^3} \\
                                                           & = \|\curl K\|_{L_2(\eps^{-1}\Dsf)^3}\|\nabla \blue \hat K_\eps^* \black\|_{L_2(\eps^{-1}\Dsf)^3}.
        \end{split}
   \een
This implies $\|\nabla \blue \hat K_\eps^* \black\|_{L_2(\VR^3)^3} \le C$ for all $\eps > 0$.  
        Now fix $\tilde \eps >0$ and let $\eps \in (0,\tilde \eps)$. Then we obtain from \eqref{E:hat_K_eps2} (by extending $K$ and $\blue \hat K_\eps^* \black$ by zero outside of $\eps^{-1}\Dsf$), 
\ben\label{E:hat_K_eps2_}
\int_{\VR^3} \curl \blue \hat K_\eps^* \black \cdot \curl \varphi \;dx = \int_{\VR^3} \curl K \cdot \curl \varphi \;dx \quad \text{ for all } \varphi\in H^1_0(\tilde \eps^{-1}\Dsf), \; \Div(\varphi)=0.
        \een
        Let $(\eps_n)$ be a null-sequence. In view of the boundedness of $(\blue \hat K_{\eps_n}^* \black)$ in $\dot{BL}(\VR^3)$, we can extract a subsequence (denoted the same) and find $\tilde K \in \dot{BL}(\VR^3)$, such that $\partial \blue \hat K_{\eps_n}^* \black \tow  \partial  \tilde K$ and thus also  $\curl \blue \hat K_{\eps_n}^* \black \tow  \curl  \tilde K$ weakly in $L_2(\VR^3)^3$. Therefore, selecting $\eps = \eps_n$ in  \eqref{E:hat_K_eps2_} we can pass to the limit $n \to \infty$ to obtain 
\ben\label{E:hat_K_eps3}
\int_{\VR^3} \curl \tilde K \cdot \curl \varphi \;dx = \int_{\VR^3} \curl K \cdot \curl \varphi \;dx \quad \text{ for all } \varphi\in H^1_0(\tilde \eps^{-1}\Dsf), \; \Div(\varphi)=0.
        \een
        Since $\Div(\blue \hat K_\eps^* \black)=0$ for all $\eps>0$ and in view of the weak convergence $\partial \blue \hat K_{\eps_n}^* \black \tow \partial\tilde K$, it is also readily checked that $\Div(\tilde K)=0$.
        Since $\tilde \eps $ was arbitrary and since $C^\infty_c(\VR^3)/\VR$ is dense in $\dot{BL}(\VR^3)$
        it follows that  \eqref{E:hat_K_eps3} holds for test functions in $BL(\VR^3)^3$ from 
        which we conclude that $\tilde K = K$. Therefore $\blue \hat K_\eps^* \black \tow K$ weakly in $\dot{BL}(\VR^3)^3$. 
        The strong convergence follows by testing \eqref{E:hat_K_eps2} with $\varphi = \blue \hat K_\eps^* \black$ and passing to the limit $\eps\searrow 0$. This shows that $\|\curl \blue \hat K_\eps^* \black\|_{L_2(\VR^3)^3} \to \|\curl K\|_{L_2(\VR^3)^3}$ as $\eps \searrow 0$. Since in a Hilbert space norm convergence together with weak convergence implies strong convergence we 
        finish the proof.
    \end{proof}

\begin{lemma}\label{L:Heps_K}
We have
\ben
\curl \blue H_\eps^* \black \to \curl K \quad \text{ strongly in } L_2(\VR^3)^3 \text{ as }  \eps \searrow 0. 
\een
    \end{lemma}
    \begin{proof}
        Subtracting \eqref{E:approx_Keps} from \eqref{E:limit_equation_W} and \blue introducing a zero term leads to \black
\ben\label{E:rewrite_first_approx}
\begin{split}
    \int_{\VR^3}   & (\Ca_\omega(x,\curl \blue \hat K_\eps^* \black + U_0 ) - \Ca_\omega(x,\curl H_\eps^* + U_0 ))\cdot \curl \varphi \; dx\\
                   & = \int_{\VR^3}   (\Ca_\omega(x,\curl \blue \hat K_\eps^* \black + U_0 ) - \Ca_\omega(x, \curl K + U_0 ))\cdot \curl \varphi \; dx
               \end{split}
\een
for all $\varphi \in H^1_0(\eps^{-1} \Dsf)^3$. Here we used the observation of Remark~\ref{rem:test_function_limit} and $H_0^1(\eps^{-1}\Dsf)^3\subset \dot{BL}(\VR^3)^3$. Now we test this equation with $\varphi = \blue \hat K_\eps^* \black - H_\eps^* \in H^1_0(\eps^{-1}\Dsf)\subset \dot{BL}(\VR^3)^3$, use the monotonicity of $\Ca_\omega$ and H\"older's inequality:
\ben\label{E:strong_Keps_Heps}
\begin{split}
    C\|\curl(\blue \hat K_\eps^* \black & - H_\eps^*) \|_{L_2(\VR^3)^3}^2 \\
                         & \le 
                         \int_{\VR^3}   (\Ca_\omega(x,\curl\blue \hat K_\eps^* \black + U_0 ) - \Ca_\omega(x,\curl H_\eps^* + U_0 ))\cdot \curl (\blue \hat K_\eps^* \black - H_\eps^* ) \; dx
                  \\
                         & \stackrel{\eqref{E:rewrite_first_approx}}{=} \int_{\VR^3}   (\Ca_\omega(x,\curl \blue \hat K_\eps^* \black + U_0 ) - \Ca_\omega(x, \curl  K + U_0 ))\cdot \curl (\blue \hat K_\eps^* \black - H_\eps^* ) \; dx\\
                          & \le \|\curl(\blue \hat K_\eps^* \black - K) \|_{L_2(\VR^3)^3} \|\curl(\blue \hat K_\eps^* \black - H_\eps^* ) \|_{L_2(\VR^3)^3}.
\end{split}
\een
It follows from Lemma~\ref{L:projection_convergence}, we have $\curl\blue \hat K_\eps^* \black \to \curl K$ strongly in $L_2(\VR^3)^3$. Therefore \eqref{E:strong_Keps_Heps} implies $\curl (\blue \hat K_\eps^* \black - H_\eps^* ) \to 0$ strongly in $L_2(\VR^3)^3$ and therefore also $\|\curl(H_\eps^*  - K)\|_{L_2(\VR^3)^3} \le \|\curl(H_\eps^*  - \blue \hat K_\eps^* \black)\|_{L_2(\VR^3)^3} + \|\curl(\blue \hat K_\eps^* \black - K)\|_{L_2(\VR^3)^3} \to 0$ as $\eps \searrow 0$.
\end{proof}


We now prove that $\curl \blue (H_\eps^* - K_\eps^*) \black \to 0$ strongly in $L_2(\VR^3)^3$.
    \begin{lemma}\label{L:Heps_Keps}
 Assume there are $\delta >0$ and $\alpha >0$, such that $\curl u_0\in C^\alpha(\overline{B_\delta(z)})^3$. Then we have
\ben
\curl \blue(H_\eps^* - K_\eps^*) \black \to 0 \quad \text{ strongly in }  L_2(\VR^3)^3 \; \text{ as } \eps \searrow 0.
\een
    \end{lemma}
    \begin{proof}
        Subtracting \eqref{E:per_eps} and \eqref{E:approx_Keps} we obtain
\ben
\begin{split}
    \int_{\VR^3} &  (\Ca_\omega(x,\curl K_\eps^*  + \curl u_0(x_\eps) ) - \Ca_\omega(x,\curl H_\eps^* + U_0 ))\cdot \curl \varphi \; dx \\
                 & + \int_{\VR^3}   (\Ca_\omega(x,U_0 ) - \Ca_\omega(x,\curl u_0(x_\eps) ))\cdot \curl \varphi \; dx \\
                &  = - \int_{\omega}(a_1 (\curl u_0 (x_\eps) ) - a_2 (\curl u_0(x_\eps) )) \cdot \curl \varphi  \blue-\black (a_1 (U_0 ) - a_2 (U_0 )) \cdot \curl \varphi \;dx
\end{split}
\een
for all $\varphi \in H^1_0(\eps^{-1} \Dsf)$ where we recall the notation $x_\eps = \eps x$. We want to use the monotonicity of $\Ca_\omega$ and therefore we rewrite the previous equation as follows
\ben\label{E:diff_Heps_Keps}
\begin{split}
    \int_{\VR^3} &  (\Ca_\omega(x,\curl K_\eps^*  + \curl u_0(x_\eps) ) - \Ca_\omega(x,\curl H_\eps^* + \curl u_0(x_\eps) ) ))\cdot \curl \varphi \; dx \\
    = &  \underbrace{  - \int_{\VR^3}  ( (\Ca_\omega(x,\curl H_\eps^* + \curl u_0(x_\eps) ) - (\Ca_\omega(x,\curl H_\eps^* + U_0 )\cdot \curl \varphi \; dx}_{=:I_1(\eps,\varphi)} \\
                 & \underbrace{- \int_{\VR^3}   (\Ca_\omega(x,U_0 ) - \Ca_\omega(x,\curl u_0(x_\eps) ))\cdot \curl \varphi \; dx}_{=:I_2(\eps,\varphi)} \\
                 &  \underbrace{ - \int_{\omega}(a_1 (\curl u_0 (x_\eps) ) - a_2 (\curl u_0(x_\eps) )) \cdot \curl \varphi \;dx  \blue - \black (a_1 (U_0 ) - a_2 (U_0 )) \cdot \curl \varphi \;dx}_{=:I_3(\eps,\varphi)}
\end{split}
\een
Now the $a_i$ are Lipschitz continuous and $\curl u_0 \in C^{\alpha}(\overbar{B_\delta(z)})^3$ with $\alpha,\delta>0$, we immediately obtain that $|I_3(\eps,\varphi)| \le C\eps^\alpha \|\curl \varphi\|_{L_2(\VR^3)^3} $ for a suitable constant $C>0$. We now show that also $|I_1(\eps,\varphi)+I_2(\eps,\varphi)|\le C(\eps)\|\curl \varphi\|_{L_2(\VR^3)^3}$ and $C(\eps)\to 0$ as $\eps \searrow 0$. We write for  arbitrary $r\in (0,1$),
\ben\label{E:sum_I1_I2}
\begin{split}
    I_1(\eps,\varphi)+I_2(\eps,\varphi)  = &  - \int_{B_{\eps^{-r}}}  ( (\Ca_\omega(x,\curl H_\eps^* + \curl u_0(x_\eps) ) - (\Ca_\omega(x,\curl H_\eps^* + U_0 )\cdot \curl \varphi \; dx \\
                                           & - \int_{B_{\eps^{-r}}}   (\Ca_\omega(x,U_0 ) - \Ca_\omega(x,\curl u_0(x_\eps) ))\cdot \curl \varphi \; dx \\
                                           &  - \int_{\VR^3\setminus B_{\eps^{-r}}}  ( (\Ca_\omega(x,\curl H_\eps^* + \curl u_0(x_\eps) ) - (\Ca_\omega(x, \curl u_0(x_\eps) )\cdot \curl \varphi \; dx \\
                                           &  + \int_{\VR^3\setminus B_{\eps^{-r}}}   ((\Ca_\omega(x,\curl H_\eps^* + U_0 ) - \Ca_\omega(x, U_0 ))\cdot \curl \varphi \; dx.
\end{split}
\een
Now we can estimate the right hand side of \eqref{E:sum_I1_I2} using the Lipschitz continuity of $a_i$ (see Assumption~\ref{A:nonlinearity}(ii)) as follows
\ben
\begin{split}
    |I_1(\eps,\varphi) + I_2(\eps,\varphi)| & \le  2C\int_{B_{\eps^{-r}}} \blue |U_0-\curl u_0(x_\eps) | |\curl \varphi | \black \;dx   
                                                   + 2C\int_{\VR^3\setminus B_{\eps^{-r}}} \blue |\curl H_\eps^* | \, |\curl \varphi| \black\;dx \\
                                        &\leq C \int_{B_{\eps^{-r}}} \blue |x_\eps|^\alpha |\curl \varphi| \black \mbox dx           + 2C\int_{\VR^3\setminus B_{\eps^{-r}}}\blue |\curl H_\eps^* |\,|\curl \varphi| \black \;dx \\                                        
                                                 & \le \eps^{-r \alpha} \eps^\alpha {\color{blue}\eps^{-3r/2} } C  \|\curl \varphi\|_{L_2(\VR^3)^3} + 2C \|\curl H_\eps^*\|_{L_2(\VR^3\setminus B_{\eps^{-r}})^3}\|\curl \varphi\|_{L_2(\VR^3\setminus B_{\eps^r})^3}
\end{split}
\een
For $r$ sufficiently close to $0$, we have $\eps^{-r\alpha} \eps^\alpha {\color{blue}\eps^{-3r/2} = \eps^{\alpha - r(\frac{3}{2} +\alpha)} } \to 0$. Moreover, by the triangle inequality we have 
\ben
\|\curl H_\eps^* \|_{L_2(\VR^3\setminus B_{\eps^{-r}})} \le \|\curl (H_\eps^*- K) \|_{L_2(\VR^3\setminus B_{\eps^{-r}})} + \|\curl K \|_{L_2(\VR^3\setminus B_{\eps^{-r}})}.
\een
The first term on the right hand side goes to zero in view of  Lemma~\ref{L:Heps_K}. The second
term goes to zero since $\curl K\in L_2(\VR^3)^3$ thus $\|\curl K\|_{L_2(\VR^3\setminus B_{\eps^{-r}})^3} \to 0$ as $\eps \searrow 0$. Using $K_\eps^*  - H_\eps^*$ as test function in \eqref{E:diff_Heps_Keps}, using the monotonicity of $\Ca_\omega$ and employing $|I_1(\eps,\varphi)+I_2(\eps,\varphi)+I_3(\eps,\varphi)| \le C(\eps)\|\curl \varphi\|_{L_2(\VR^3)^3}$ with $C(\eps) \to 0$ as $\eps \searrow 0$, shows the result.
    \end{proof}
    Combining Lemma~\ref{L:Heps_K} and Lemma~\ref{L:Heps_Keps} proves the Main Theorem~\ref{maintheorem_1}(ii). \hfill $\blacksquare$

    \section{The topological derivative}\label{sec:topological_derivative}
In this section we show that the hypotheses of Theorem~\ref{thm:diff_lagrange} are satisfied for the Lagrangian $G$ given by \eqref{eq:lagrange_scale_eps}.

Let $\ell(\eps):= |\omega_\eps|$, and introduce the Lagrangian $ G:[0,\tau]\times H_0(\Dsf,\curl) \times H_0(\Dsf,\curl) \to \VR$ associated with the perturbation $\omega_\eps$ by 
\ben\label{eq:lagrange_scale_eps}
G(\eps, u, p) :=  \int_{\blue \Omega_g\black} \blue | \curl(u) - B_d|^2 \black \; dx +  \int_{\Dsf} \Ca_{\Omega_\eps}(x, \curl u) \cdot \curl p  \; dx - \langle F,p\rangle.
\een
Here, the operator $\Ca_{\Omega_\eps}$ is defined according to \eqref{eq_AOmega} with $\Omega_\eps = \Omega \cup \omega_\eps$. It is clear from Assumption~\ref{A:nonlinearity} that the Lagrangian $G$ is $\ell$-differentiable in the sense of Definition~\ref{D:l_differentiable} with $X=Y=V(\Dsf)$ and $\ell(\eps):= |\omega_\eps|$.

\begin{maintheorem}
    Let Assumption~\ref{A:nonlinearity} be satisfied. Let $\Omega \subset D$ open and $u_0$ the solution to \eqref{E:weakformulation} and $p_0$ the solution to \eqref{eq_adjoint}. Let $z \in  \Dsf \setminus \overbar \Omega$, such that $z\not\in (\Omega_1\cup\Omega_2\cup \Omega_g)$. Further assume that $\curl u_0\in C^{\alpha}(\overline{B_\delta(z)})^3$ for some $\delta >0$ and $0 <\alpha<1$ 
    and also $\curl p_0\in C(\overline{B_\delta(z)})^3\cap L_\infty(\Dsf)^3$. 
    \begin{itemize}
    \item[(a)] Then the assumptions of Theorem~\ref{thm:diff_lagrange} are satisfied for the Lagrangian $G$ given by \eqref{eq:lagrange_scale_eps} and hence the topological derivative at $z\in \Dsf \setminus \overbar \Omega$ is given by 
\ben \label{eq_td}
dJ(\Omega)(z) = \partial_{\ell} G(0, u_0, p_0) + R_1 ( u_0, p_0) + R_2(u_0,p_0) 
  \een
\item[(b)] 
We have 
\ben \label{eq_dlG}
\partial_{\ell} G(0, u_0, p_0) = ((a_1 (U_0) - a_2 (U_0)) \cdot P_0
\een
and 
\ben\label{E:R_term}
\begin{split}
    R_1( u_0, p_0)   = \frac{1}{|\omega|} \bigg(\int_{\VR^3} 
    \big[\Ca_\omega(x, \curl K+U_0) - \Ca_\omega(x,U_0) - \partial_{u} \Ca_\omega(x,U_0 )& (\curl K) 
               \big] 
           \cdot P_0\;dx\bigg)
\end{split}
\een
and 
\ben
R_2(u_0,p_0)=          \frac{1}{|\omega|}\int_{\omega} \left[\partial_{u} a_1 (U_0) - \partial_{u} a_2 (U_0) \right] (\curl K) \cdot P_0   \;dx 
\een
where $U_0 := \curl u_0(z)$, $P_0:= \curl p_0(z)$ and $\Ca_\omega(x,y):= a_1(y)\chi_\omega(x) + a_2(y)\chi_{\VR^3\setminus \omega}(x)$,
and $K$ is the unique solution to \eqref{E:limit_equation_W} and $p_0\in V(\Dsf)$ solves
\ben \label{eq_adjoint}
\int_\Dsf \partial_{u}\Ca_\Omega(x,\curl u_0)(\curl \varphi)\cdot \curl p_0 \;dx = -  \int_\Dsf 2(\curl u_0- B_d)\cdot \curl \varphi\;dx
\een
\end{itemize}
\end{maintheorem}

\begin{remark} \label{rem_z}
    \begin{itemize}
        \item We restrict ourselves to the case where $z \in D \setminus \overbar \Omega$ without loss of generality. However, the exact same proof can be conducted in the case where $z \in  \Omega$ and $z\not\in (\Omega_1\cup\Omega_2\cup\Omega_g)$. In that case, the formula for the topological derivative is obtained by just switching the roles of $a_1$ and $a_2$ in the theorem above (in particular also in the definition of $\Ca_\omega$).

        \item Also the case where $z\in \Omega_1\cup\Omega_2\cup\Omega_g$ can be dealt with in a similar manner. 
Indeed the derivation of \cite{a_GanglSturm_2019a} shows that for instance if $z\in \Omega_g$ an 
additional term $\int_{\VR^3}|\nabla K|^2\;dx $ in $dJ(\Omega)(z)$ appears. The case $z\in \Omega_1$ and/or $z\in \Omega_2$ have to be treated separately since in this case the right hand side 
 $F$ becomes domain dependent.

    \item The assumption $z=0$ is without loss of generality, too. In the general case, this situation can be obtained by a simple change of the coordinate system.
\end{itemize}
\end{remark}

\subsection{Computation of $R_1(u_0,p_0)$ and $R_2(u_0,p_0)$}

It remains to check that the limits of $R_1(u_0,p_0)$ and $R_2(u_0,p_0)$ exist. For this we use Assumption~\ref{A:nonlinearity}(i)-(iii). Using the change of variables $T_\eps(x) = \eps x $ and the definition $\ell(\eps)=|\omega_\eps|=\eps^3|\omega|$, we have
\ben
\begin{split}
    R_1^\eps (u_0,p_0) & =\frac{1}{\ell(\eps)} \int_0^1 \int_{\Dsf} \left(\partial_u \Ca_{\eps}(x, \curl  (su_\eps  + (1-s)u_0)) - \partial_u \Ca_{\eps}(x, \curl  u_0 ) \right)(\curl  (u_\eps-u_0))  \cdot \curl  p_0\;dx\;  ds \\
                       &  + \frac{1}{\ell(\eps)} \int_{\Omega_g} |\curl  (u_\eps  - u_0)|^2 \; dx\\
                       & =  \underbrace{\frac{1}{|\omega|} \int_0^1 \int_{\VR^3} \left(\partial_u \Ca_{\omega}(x,  s\curl  K_\eps^*    + \curl  u_0(x_\eps) ) - \partial_u \Ca_{\omega}(x, \curl  u_0(x_\eps) ) \right)(\curl  K_\eps^* )  \cdot \curl  p_0 (x_\eps)\;dx\;  ds}_{=:I_\eps} \\
                   &  + \underbrace{\frac{1}{|\omega|} \int_{\eps^{-1}\Omega_g} |\curl  K_\eps^*  |^2 \; dx}_{=:II_\eps}\\
                   & \to \frac{1}{|\omega|} \int_0^1 \int_{\VR^3} \left(\partial_u \Ca_{\omega}(x,  s\curl  K   + U_0) - \partial_u \Ca_{\omega}(x, U_0 ) \right)(\curl  K )  \cdot P_0\;dx\;  ds.
\end{split}
\een
Since  $\curl  K_\eps^*  \to \curl  K$ strongly in $L_2(\VR^3)^3$ as $\eps\searrow 0$ 
and since $\eps^{-1}\Omega_g$ goes to "infinity" because $z\not\in \Omega_g$ it readily follows that
$II_\eps \to 0$ as $\eps\searrow 0$. To see the convergence of the first term, we may write $I_\eps$as follows
\begin{align*}
     \int_0^1 \int_{\VR^3} (\partial_u & \Ca_{\omega}(x,  s\curl  K_\eps^*     + \curl  u_0(x_\eps) ) - \partial_u \Ca_{\omega}(x, \curl  u_0(x_\eps) ) )(\curl  K_\eps^* )  \cdot \curl  p_0 (x_\eps)\;dx ds = \\
                & \int_0^1 \int_{\VR^3}(\partial_u \Ca_{\omega}(x,  s\curl  K_\eps^*    + \curl  u_0(x_\eps) ) - \partial_u \Ca_{\omega}(x,  s\curl  K   + \curl  u_0(x_\eps) ) )(\curl  K_\eps^* )  \cdot \curl  p_0 (x_\eps) \;dx ds \\
                &+ \int_0^1 \int_{\VR^3}(\partial_u \Ca_{\omega}(x,  s\curl  K   + \curl  u_0(x_\eps) ) - \partial_u \Ca_{\omega}(x, \curl  u_0(x_\eps) ) )(\curl  (K_\eps^* -K))  \cdot \curl  p_0 (x_\eps) \;dx ds\\
                                                          & +\int_0^1 \int_{\VR^3}(\partial_u \Ca_{\omega}(x,  s\curl  K   + \curl  u_0(x_\eps) ) - \partial_u \Ca_{\omega}(x, \curl  u_0(x_\eps) ) )(\curl  K)  \cdot \curl  p_0 (x_\eps)\;dx ds.
\end{align*}
Using Assumption~\ref{A:nonlinearity}(iii) and $\curl  p_0 \in L^\infty(\Dsf)^3$, we see that the 
absolute value of the first and second term on the right hand side can be bounded by $C\|\curl (K_\eps^* -K)\|_{L_2(\VR^3)^3}\|\curl  K\|_{L_2(\VR^3)^3}$ and hence using $\curl  K_\eps^*  \to \curl  K$ in $L_2(\VR^3)^3$ as $\eps\searrow 0$ they disappear in the limit. The last term converges to the desired limit by using Lebesgue's dominated convergence theorem. 
Using the fundamental theorem, we obtain the expression in \eqref{E:R_term}. Similarly, using \eqref{rem_aiii}, the continuity of $\curl  u_0$ and $\curl  p_0$ at $z$, the continuity of $\partial_ua_1,\partial_ua_2$, and again $\curl  K_\eps^*  \to \curl  K$ strongly in $L_2(\VR^3)^3$, we obtain by Lebesgue's dominated convergence theorem
\ben
\begin{split}
    R_2^\eps (u,p) & = \frac{1}{\ell(\eps)} \int_{\omega_\eps } (\partial_u a_1(\curl  u_0) - \partial_u a_1(\curl  u_0))(\curl  (u_\eps-u_0))  \cdot \curl  p_0\;dx\\
                   & = \frac{1}{|\omega|}\int_{\omega} (\partial_u a_1(\curl  u_0(x_\eps)) - \partial_u a_2(\curl  u_0(x_\eps)))(\curl  K_\eps^* )  \cdot \curl  p_0(x_\eps)\;dx \\
                   & \to \frac{1}{|\omega|}\int_{\omega} (\partial_u a_1(U_0) - \partial_u a_2(U_0))(\curl  K)  \cdot P_0\;dx.
\end{split}
\een

Therefore all Hypotheses of Theorem~\ref{thm:diff_lagrange} are satisfied. This finishes the 
proof of our Main Theorem 2.

\section{Numerical realization}\label{sec:numerics}
Formula \eqref{eq_td} together with \eqref{eq_dlG}--\eqref{eq_adjoint} states the topological derivative for problem \eqref{defJ}--\eqref{E:weakformulation} at a single spatial point $z$. Note that the evaluation of the topological derivative involves the solution of problem \eqref{E:limit_equation_W}, which in turn depends on the point $z$ via the vector $U_0 = \curl(u_0)(z)$. When using the topological derivative \eqref{eq_td} in a numerical optimization algorithm, it has to be evaluated at every point in the design area in every iteration of the algorithm. Therefore, a direct evaluation of \eqref{eq_td} is unfeasible and an efficient technique for numerical approximation is indispensable. In this section, we show a way to approximately evaluate formula \eqref{eq_td} by first precomputing certain values in an offline phase and looking them up and interpolating them during the online phase of the optimization algorithm. We proceed in an analogous way to \cite[Sec. 7]{AmstutzGangl2019}.

For this, we need the following additional assumption:
\begin{assumption} \label{assump_a_R} 
    \begin{enumerate}
     \item[(i)] For all orthogonal matrices $R \in \VR^{3\times3}$ and all $y \in \VR^3$, it holds that 
     \ben
        a_i(R y ) = R \, a_i(y) \quad \mbox{for } i=1,2.
    \een
     \item[(ii)] The inclusion is the unit ball: $\omega = B_1(0)$.
    \end{enumerate}
\end{assumption}

We will show a concrete application that satisfies this assumption in Section \ref{sec:machine}.
We note that the topological derivative \eqref{eq_td} depends on the spatial point $z$ only via $U_0$, $P_0$ and $K=K_{U_0}$. Let us make this dependence more clear by introducing the notation
\begin{align}
    dJ(\Omega)&(U_0, P_0) := ((a_1 (U_0) - a_2 (U_0)) \cdot P_0  \label{dJOmega_U0P0_1}\\
&+ \frac{1}{|\omega|} \bigg(\int_{\VR^3} 
    \big[\Ca_\omega(x, \curl K_{U_0}+U_0) - \Ca_\omega(x,U_0) - \partial_{u} \Ca_\omega(x,U_0 ) (\curl K_{U_0}) 
               \big] 
           \cdot P_0\;dx\bigg)\label{dJOmega_U0P0_2}\\
&+ \frac{1}{|\omega|}\int_{\omega} \left[\partial_{u} a_1 (U_0) - \partial_{u} a_2 (U_0) \right] (\curl K_{U_0}) \cdot P_0   \;dx .\label{dJOmega_U0P0_3}
\end{align}

\begin{remark}
Recall that $e_i$, $i=1,2,3$, denotes the $i$-th unit vector in the Cartesian coordinate system in $\VR^3$. For every vector $W \in \VR^3$ there exists an orthogonal rotation matrix $R_W$ such that $W = \blue |W| \black R_W e_1$. 
\end{remark}

The next result will allow us to introduce an efficient strategy for the approximate evaluation of the topological derivative $dJ(U_0, P_0)$ for any $U_0, P_0 \in \VR^3$.

\begin{maintheorem} \label{theo_dJ_U0P0}
    Let Assumption~\ref{assump_a_R} hold and $U_0,P_0\in \VR^3$. Then it holds:
    \begin{enumerate}
     \item[(i)] the mapping $P\mapsto dJ(\Omega)(U_0, P)$ is linear on $\VR^3$,
     \item[(ii)] $dJ(\Omega)(R^\top U_0, R^\top P_0) = dJ(\Omega)(U_0, P_0)$ for all orthogonal matrices $R \in \VR^{3\times 3}$.
     \item[(iii)]  Write $U_0 = \blue |U_0| \black R_{U_0} e_1$ and $P_0 = \blue |P_0| \black R_{P_0} e_1$ for some orthogonal matrices $R_{U_0},R_{P_0}\in \VR^{3\times 3}$ and set $(c_1, c_2, c_3)^\top := \blue |P_0| \black R_{U_0}^\top R_{P_0} e_1$. Then we have
    \ben
    dJ(\Omega)(U_0, P_0) = c_1 dJ(\Omega)(\blue |U_0| \black e_1, e_1) + c_2 dJ(\Omega)(\blue |U_0| \black e_1, e_2) + c_3 dJ(\Omega)(\blue |U_0| \black e_1, e_3).
    \een
    \end{enumerate}
\end{maintheorem}

\begin{corollary} \label{cor_dJ_U0P0} Let Assumption \ref{assump_a_R} hold. Suppose that the values $dJ(\Omega)(t e_1, e_i)$, $i=1,2,3$ are given for all $t \in [t_{min}, t_{max}]$ with $0\leq t_{min} < t_{max}$. Then, for all $U_0 \in \VR^3$ with $t_{min} \leq \blue |U_0| \black \leq t_{max}$ and all $P_0 \in \VR^3$ it holds
    \ben
    dJ(\Omega)(U_0, P_0) = c_1 dJ(\Omega)(\blue |U_0| \black e_1, e_1) + c_2 dJ(\Omega)(\blue |U_0| \black e_1, e_2) + c_3 dJ(\Omega)(\blue |U_0 | \black e_1, e_3). \label{eq_dJ_U_P}
    \een
with $(c_1, c_2, c_3)^\top = \blue|P_0| \black  R_{U_0}^\top R_{P_0} e_1$.
\end{corollary}
We first prove the following properties of the solution mapping $W\mapsto K_W$, where 
$K_W$ denotes the unique solutions in $\dot{BLC}(\VR^3)$ to \eqref{E:limit_equation_W} with $U_0$ being replaced by $W\in \VR^3$.

\begin{lemma} \label{lem_K_rot}
Let Assumption~\ref{assump_a_R} hold. Let $W \in \VR^3$, $R \in \VR^{3\times 3}$ orthogonal.
Then the following relations hold:
    \begin{align}
        K_{R^\top W}(x) &= R^\top K_W(Rx) + \nabla \eta, \\
        \curl (K_{R^\top W})(x) &= R^\top (\curl (K_W))(Rx). \label{eq_curlK_rot}
    \end{align}
\end{lemma}

\begin{proof}
    To see the first identity, we perform the change of variables $y=\phi(x) = Rx$ in \eqref{E:limit_equation_W} with $U_0$ replaced by $W$. Noting that the chain rule yields
    \ben \label{eq_trafo_curl}
        \curl_y(K) \circ \phi = R \curl_x(R^T (K\circ \phi)),
    \een
    where we used $\mbox{det}(R) = 1$, we get for \eqref{E:limit_equation_W}
    \begin{align*}
        \int_{\VR^3} \bigg( \Ca_{\phi^{-1}(\omega)} (x, R \curl_x (\tilde K) + W) &- \Ca_{\phi^{-1}(\omega)} (x, W) \bigg) \cdot R \curl_x(\tilde \varphi) =\\
        &-\int_{\phi^{-1}(\omega)}\bigg( a_1(W)-a_2(W) \bigg) \cdot R \curl_x(\tilde \varphi).
    \end{align*}    
    Here we used the notation $\tilde K = R^\top (K_W\circ \phi)$ and $\tilde \varphi = R^\top (\varphi \circ \phi)$. Using Assumption~\ref{assump_a_R}, this can be rewritten as
    \begin{align*}
        \int_{\VR^3} \bigg( \Ca_{\omega} (x, \curl_x (\tilde K) + R^\top W) &- \Ca_{\omega} (x, R^\top W) \bigg) \cdot  \curl_x(\tilde \varphi) =
        -\int_{\omega}\bigg( a_1(R^\top W)-a_2(R^\top W) \bigg) \cdot  \curl_x(\tilde \varphi).
    \end{align*}  
    Since $K_{R^\top W}$ is the unique solution in $\dot{BLC}(\VR^3)$ to the problem above, we conclude that $R^\top (K_W\circ \phi) = \tilde K = K_{R^\top W}$ in $\dot{BLC}(\VR^3)$. Finally this relation together with \eqref{eq_trafo_curl} yields
    \ben
        \curl_x(K_{R^\top W}) = \curl_x(R^\top (K_W \circ \phi)) = R^\top \curl_y(K_W)\circ \phi.
    \een
\end{proof}

\begin{proof}[Proof of Main Theorem~\ref{theo_dJ_U0P0}]
    The first statement can be seen directly from \eqref{dJOmega_U0P0_1}--\eqref{dJOmega_U0P0_3}. The second result follows immediately by Assumption~\ref{assump_a_R} using \eqref{eq_curlK_rot} noting that Assumption B(i) implies $\partial_u a_i(R y )(R z) = R \partial_u a_i(y )(z)$ for $R \in \VR^{3\times 3}$ orthogonal and $y,z \in \VR^3$ and $i=1,2$.

    Using the representations $U_0 = \blue |U_0|\black  R_{U_0} e_1$ and $P_0 = \blue |P_0| \black  R_{P_0} e_1$ and item (ii), we get
\begin{align}
    dJ(\Omega)(U_0, P_0) = dJ(\Omega) (\blue |U_0| \black  R_{U_0} e_1, \blue |P_0| \black  R_{P_0} e_1) = dJ(\Omega)(\blue |U_0| \black  e_1, \blue|P_0| \black R_{U_0}^\top R_{P_0} e_1). \nonumber  
\end{align}
The result now follows from the definition $(c_1, c_2, c_3)^\top = \blue |P_0| \black  R_{U_0}^\top R_{P_0} e_1$ and the linearity of $dJ(\Omega)(\cdot, \cdot)$ in the second argument (cf. item (i)).
\end{proof}

Our proposed strategy now consists in first precomputing $dJ(\Omega)(t e_1, e_i)$, $i=1,2,3$ for a range of values of $t= \blue |U_0 | \black  = \blue | \curl u_0(z) | \black $ between a minimum value $t_{min}= 0$ and a maximum value $t_{max}$ in an offline stage. During the optimization, the values of $dJ(\Omega)(t e_1, e_i)$ for any $t \in [t_{min},t_{max}]$ can be approximated by interpolation and the topological derivative can be (approximately) evaluated with the help of Corollary~\ref{cor_dJ_U0P0}. In practical applications, often reasonable values for $t_{max}$ are known.

For the precomputation of the values $dJ(\Omega)(t e_1, e_i)$ for a fixed $t \in [t_{min}, t_{max}]$, problem \eqref{E:limit_equation_W} has to be solved with $U_0 := t e_1$. For the numerical solution of \eqref{E:limit_equation_W} recall that the solution $H_\eps$ to \eqref{E:approx_Keps} is a good approximation of $K$ for small $\eps>0$ due to Lemma \ref{L:Heps_K}. Moreover, it can be shown in an analogous way to Lemma \ref{L:Heps_K} that for $B:=B_R(0)$ with $R$ such that $B \subset \Dsf$, the solution $\tilde H_\eps \in V(\eps^{-1} B)$ to 
\ben \label{eq_HepsTilde}
\begin{split}
    \int_{\eps^{-1}B}  & (\Ca_\omega(x,\curl \tilde H_\eps + U_0 ) - \Ca_\omega(x, U_0 ))\cdot \curl \varphi \; dx\\
              & = - \int_{\omega}(a_1 (U_0) - a_2 (U_0)) \cdot \curl \varphi \;dx \quad \text{ for all } \varphi \in V(\eps^{-1} B).
\end{split}
\een
satisfies $\curl \tilde H_\eps \rightarrow \curl K$ strongly in $L_2(\VR^3)$. Motivated by this observation, one may solve \eqref{eq_HepsTilde} with $B = B_1(0)$ and a comparatively small value for $\eps$, e.g. $\eps = 1/1000$, as a good approximation to \eqref{E:limit_equation_W}.

\section{Application to electrical machines}\label{sec:machine}
In this section we show a real-world application where the setting of this paper applies. We consider the topology optimisation of an electrical machine in the setting of three-dimensional magnetostatics with nonlinear material behavior.

\subsection{Physical modeling}
\begin{figure}
    \begin{center}
        \includegraphics[width=0.8 \textwidth]{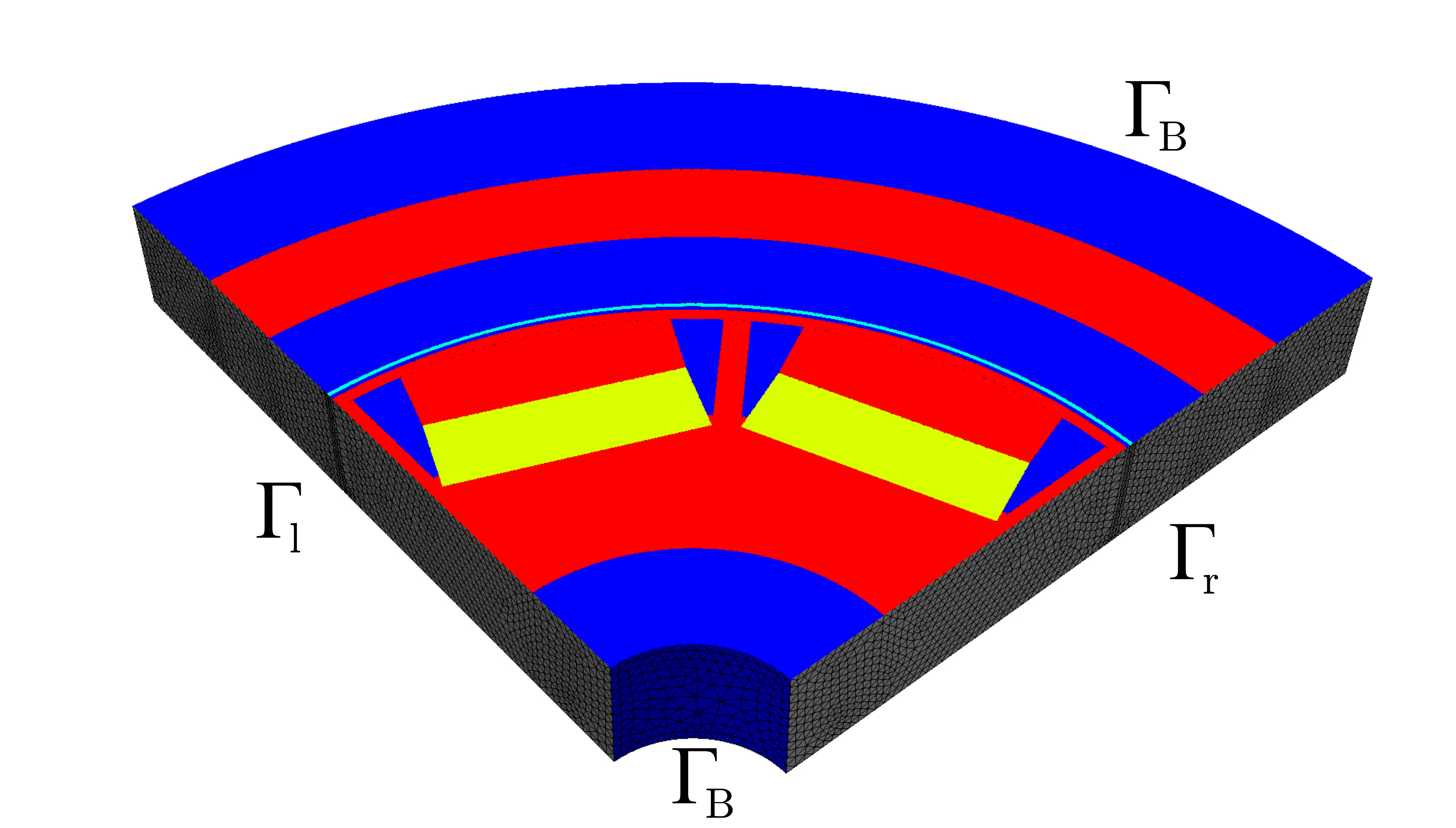}
    \end{center}
    \caption{Three-dimensional model of electrical machine $\Dsf$ with ferromagnetic subdomain $\Omega$ (red), permanent magnet regions $\Omega_2$ (yellow), air subdomain $\Dsf \setminus (\Omega \cup \Omega_2)$ (blue) with air gap region $\Omega_g$ (turquoise), lateral boundaries $\Gamma_l$, $\Gamma_r$ and inner and outer boundaries $\Gamma_B$.}
    \label{fig_motor3D}
\end{figure}

 The magnetostatic regime is a low frequency approximation to the full Maxwell equations where all quantities are assumed to be time-independent and where one only considers the magnetic equations
\ben \label{maxwellmagstat}
	\curl H = J_i \quad \mbox{ and } \quad \Div B = 0.
\een
Here, $J_i$ denotes the impressed current density, and the magnetic field intensity $H$ and the magnetic flux density $B$ are related by the nonlinear material law
\ben \label{materiallaw}
	H = \nu(\blue |B| \black )(B - M),
\een
where $\nu:\mathbb R_0^+ \rightarrow \mathbb R_0^+$ is the material-dependent magnetic relucitivity and $M$ denotes the permanent magnetization. Due to symmetry, we consider only a quarter of the machine. 
Let $\Dsf =\{(x,y,z)\in \VR^3: r_0 <\sqrt{ x^2+y^2} < r_3, -2.5 < z < 2.5 \}$ with $r_0=4$mm and $r_3=33$mm the bounded, simply connected Lipschitz domain which contains the quarter of the electric motor as depicted in Figure \ref{fig_motor3D}. We set periodic boundary conditions on the lateral boundaries $\Gamma_l$ and $\Gamma_r$, natural boundary conditions $\nu(\blue|B|\black ) B \times n = 0$ on the top and bottom and induction boundary conditions $B \cdot n = 0$ on the inner and outer parts $\Gamma_B$ of $\partial \Dsf$.

The motor consists of an inner, rotating part (the rotor) and an outer, fixed part (the stator), both containing ferromagnetic components. They are separated by a thin air gap $\Omega_g = \{(x,y,z) \in \Dsf: r_1 < \sqrt{ x^2+y^2} < r_2 \}$ with $r_1 = 19.67$mm and $r_2=19.83$mm, see the light blue area in Figure \ref{fig_motor3D}. We denote the union of all ferromagnetic subdomains by $\Omega$ which we assume to be open. The current density $J_i$ is in general supported in the coil regions $\Omega_1 \subset \Dsf \setminus \overbar \Omega$, which lie between the air gap and the stator core. The magnetization $M$ is supported in the permanent magnets $\Omega_2\subset \Dsf \setminus \overbar \Omega$. In this particular application, which was also treated in a two-dimensional setting in \cite{AmstutzGangl2019, a_GALALAMEST_2015a}, we assume the currents to be switched off, i.e., $J_i=0$ and therefore treat $\Omega_1$ as air.

The magnetic reluctivity $\nu$ is equal to a constant $\nu_0 = 10^7/(4 \pi)$ in the air and coil subdomains of the computational domain, a constant $\nu_{m}$ close to $\nu_0$ in the permanent magnet regions $\Omega_2$ and is given by a nonlinear function $\nuh$ in the ferromagnetic subdomain $\Omega$. For more compact presentation we assume $\nu_{m} = \nu_0$. Moreover, we assume that $\nuh$ has the following properties:
\begin{assumption}\label{B:nonlinearity}
	We assume that the magnetic reluctivity function $\nuh:\VR_0^+ \rightarrow \VR^+$ satisfies:
	\begin{itemize}
		
		\item[(i)] The mapping $s \mapsto \nuh(s)s$, is strongly monotone, i.e. there is a constant $\underline \nu$ such that
			\ben
				\left( \nuh(s)s - \nuh(t)t\right)(s-t) \geq \underline \nu (s-t)^2. \label{assumpA_nuMon}
			\een
		\item[(ii)] The mapping $s \mapsto \nuh(s)s$ is Lipschitz continuous, i.e. there is a constant $\overbar \nu$ such that
			\ben
				| \nuh(s)s - \nuh(t)t| \leq \overbar \nu |s-t|. \label{assumpA_nuLip}
			\een
        \item[(iii)] We assume that for $\nuh \in C^2(\VR_0^+)$, $\nuh'(0)=0$, and that there is a constant $c$ such that for all $s \in \VR_0^+$, $\nuh'(s) \leq c$ and $\nuh''(s) \leq c$.
	\end{itemize} 
\end{assumption}

The first two points of Assumption \ref{B:nonlinearity} follow from physical properties of $B-H$-curves, i.e. of the relations between magnetic flux density $B$ and magnetic field intensitiy $H$ (cf. \cite{Pechstein2004,PechsteinJuettler2006}). In practice, the function $\nuh$ is obtained by interpolation of measured values \cite{PechsteinJuettler2006}, thus the smoothness assumption in Assumption \ref{B:nonlinearity}(iii) is justified. In our numerical experiments, we chose the analytic reluctivity function
\ben
    \nuh(s) = \nu_0-(\nu_0-q_1) \, \mbox{exp}( -q_2 \, s ^{q_3}),
\een
which was also used in \cite{a_YO_2013a}, with the values $q_1=200$, $q_2=0.001$ and $q_3=6$, which satisfies all of Assumption \ref{B:nonlinearity}.

\begin{lemma} Let Assumption \ref{B:nonlinearity} hold and define $a_1(y) := \nuh(\blue|y|\black )y$ and $a_2(y) := \nu_0 y$ for $y \in \VR^3$. Then Assumption \ref{A:nonlinearity} is satisfied.
\end{lemma}

\begin{proof}
    All properties of Assumption \ref{A:nonlinearity} are clear for the linear function $a_2$. For $a_1$, items (i) and (ii) of Assumption \ref{A:nonlinearity} follow immediately from items (i) and (ii) of Assumption \ref{B:nonlinearity}, respecitively (see e.g. \cite{Pechstein2004}). Moreover, it is shown in \cite[Lemma 3.7]{AmstutzGangl2019} that Assumption \ref{B:nonlinearity}(iii) implies that $a_2$ is twice continuously differentiable, which is sufficient for Assumption \ref{A:nonlinearity}(iii).
\end{proof}

Using the ansatz $B = \curl u$ together with the Coulomb gauging condition $\Div u=0$, as well as the material law \eqref{materiallaw} and $a_1(y) := \nuh(\blue |y| \black )y$ and $a_2(y) := \nu_0 y$, we get from \eqref{maxwellmagstat} the boundary value problem
\ben \label{eq_state_motor}
    \mbox{find }u \in V: \int_\Dsf \Ca_\Omega(x, \curl u) \cdot \curl v \, dx = \int_{\Omega_2} M \cdot \curl v \, dx \quad \mbox{for all } v \in V,
\een
with the function space $V = \{ v \in H(\curl, \Dsf): u\times n = 0 \mbox{ on } \Gamma_B, u|_{\Gamma_l} = u|_{\Gamma_r}, \Div(u) = 0 \mbox{ in } \Dsf \}$ and the operator $\Ca_\Omega(x,y) = \chi_\Omega(x) a_1(y) + \chi_{D \setminus \Omega}(x) a_2(y)$. Note that we used the fact that $u \times n = 0$ on $\Gamma_B$ is sufficient for $\curl u \cdot n = B \cdot n = 0$ on $\Gamma_B$, see \cite{BuffaCostabelSheen2002, Kost1994, Pechstein2004}.

As an objective function we consider
\ben \label{eq_def_functional}
    J(\Omega) = \int_{\Omega_g} | \curl u \cdot \hat n - B_d^n |^2  \,dx
\een
where $\Omega_g$ represents the air gap of the machine, $\hat n = (x/ \sqrt{ x^2+y^2},y/ \sqrt{ x^2+y^2},0)^\top$ denotes an extension to the subdomain $\Omega_g$ of a unit normal vector field on a circular curve in the air gap and $B_d^n$ denotes the desired distribution of the normal component of the magnetic flux density $B=\curl u$ in the air gap. In our experiments, $B_d^n$ is given in cylindrical coordinates by
\ben
    B_d^n(r, \varphi, z) = - \mbox{amp}(z) \,  \mbox{sin}(4 \varphi)
\een
where $\mbox{amp}(z)$ is given by the evaluation of $(\curl u_{init} \cdot \hat n)$ at the point $(19.75, 22.5^\circ, z)$ inside the air gap $\Omega_g$. Here, $u_{init}$ denotes the solution to the PDE constraint in the initial configuration. The left picture in Figure \ref{fig_curlun} shows $\curl(u)\cdot \hat n$ as a function of the angle $\varphi \in [0,90^\circ]$ and $z \in [-2.5, 2.5]$ for a fixed value of $r=19.75$ (center of the air gap) for the initial configuration. The desired curve $B_d^n$ is depicted in the center of Figure \ref{fig_curlun}. We remark that the minimization of the objective function \eqref{eq_def_functional} yields a design of a machine which exhibits a smooth rotation pattern. Note the slight difference of objective function \eqref{eq_def_functional} to the functional \eqref{defJ} which was treated in the earlier sections. We remark, however, that all of the analysis can be performed for the given functional \eqref{eq_def_functional} in the exact same way. Note that the corresponding adjoint equation reads
\ben \label{eq_adjoint_motor}
    \int_\Dsf \partial_{u}\Ca_\Omega(x,\curl u)(\curl \varphi)\cdot \curl p \;dx = -  \int_\Dsf 2(\curl u \cdot \hat n- B_d^n) ( \curl \varphi \cdot \hat n) \;dx,
\een
where $u$ solves \eqref{eq_state_motor}.

\begin{figure}
    \begin{tabular}{ccc}
        \includegraphics[width=0.33\textwidth]{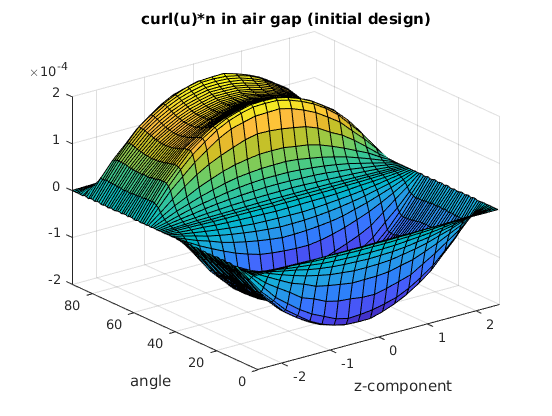}&        \includegraphics[width=0.33\textwidth]{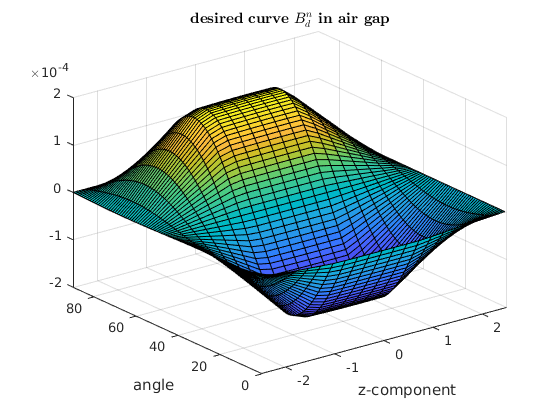}&
        \includegraphics[width=0.33\textwidth]{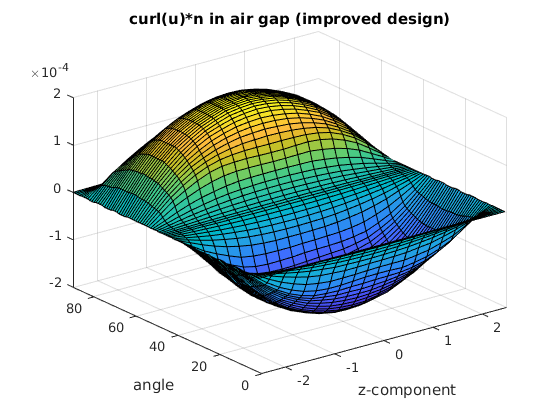}
    \end{tabular}
    \caption{$\curl(u)\cdot n_{2D}$ along the air gap for $(r, \varphi, z)$ with $r=19.75$ and $\varphi \in [0, 90^\circ]$, $z \in  [-2.5, 2.5]$. Left: Initial configuration. Center: desired curve $B_d^n$. Right: Improved configuration.}
    \label{fig_curlun}
\end{figure}


\subsection{Numerical Results}
In this section, we illustrate how the formula derived in Section \ref{sec:topological_derivative} can be applied to the optimization of the electrical machine introduced in this section. The evaluation of the topological derivative \eqref{eq_td} is done as described in Section \ref{sec:numerics}. We precomputed the values $dJ(\Omega)(t e_1, e_i)$ for $i=1,2,3$ and $t \in \{ j \, \delta_t \}_{j=0}^{40}$ with $\delta_t=0.05$ and interpolated the obtained data using quadratic B-splines in an offline phase.

For the numerical solution of the state equation \eqref{eq_state_motor}, we used second order N\'{e}d\'{e}lec finite elements, see e.g. \cite{ZaglmayrDiss2006}, \cite[Sec. 3]{SchoeberlSkript}, in the framework of the finite element software package \texttt{NETGEN/NGSolve} \cite{Schoeberl2014}. Problem \eqref{eq_state_motor} involves a divergence-free condition. In order to avoid solving a saddle point problem, we added an $L^2$-term $\int_{\Dsf} \kappa u \cdot v \, dx$ with a small constant $\kappa>0$ as regularization to the bilinear form, yielding an elliptic problem on $H(\Dsf,\curl )$. We proceeded analogously in the numerical solution of the corresponding adjoint equation \eqref{eq_adjoint_motor} and the problems for the approximation of the variation $K$ \eqref{eq_HepsTilde} in the offline phase.

We started with the initial configuration shown in Figure \ref{fig_motor3D}, where all material data is constant in $z$-direction. Figures \ref{fig_init} and \ref{fig_iter1} show the application of a one-shot topology optimisation approach to \eqref{eq_def_functional} using a level set representation. The first row of Figure \ref{fig_init} shows the level set function in the two design subdomains of interest. We start with a constant level set function $\psi_0=1$ corresponding to ferromagnetic material in all of the two design subdomains. The left column in Figures \ref{fig_init} and \ref{fig_iter1} correspond to a horizontal cut at the bottom ($z \approx -2.5$), the central column shows a cut through the center of the machine ($z=0$), and the right column a cut through the top of the machine ($z \approx 2.5$).

The second row of Figure \ref{fig_init} shows the absolute value of the magnetic flux density $\blue |B| \black  = \blue | \curl u | \black $ for the three cross sections and the third row depicts the topological derivative. Note that the topological derivative attains its most negative values in the central cross section. For better visibility, we only show the negative part of the topological derivative in the central picture.

In order to change the material in the position where the topological derivative is most negative, we set 
\ben 
    \psi_1 = (1-s) \psi_0 + s \frac{dJ(\Omega) }{\| dJ(\Omega) \|_{L^2(\Dsf)}}
\een
for an appropriately chosen value of $s$ (here: $s \approx 0.14$).

The result can be seen in Figure \ref{fig_iter1} where the design in the top and bottom cross section remain unchanged and in total four holes of air are introduced in the center. The first row of Figure \ref{fig_iter1} shows the updated level set function $\psi_1$ and the second row the corresponding distribution of the magnetic flux density in the new design.

The third picture in Figure \ref{fig_curlun} shows the distribution of $\curl u \cdot \hat n$ for the new configuration. The objective value \eqref{eq_def_functional} has dropped from $2.33*10^{-8}$ to $4.68*10^{-9}$.

\begin{figure}
    \begin{tabular}{ccc}
        \includegraphics[width=0.33\textwidth, trim=250 0 250 0, clip]{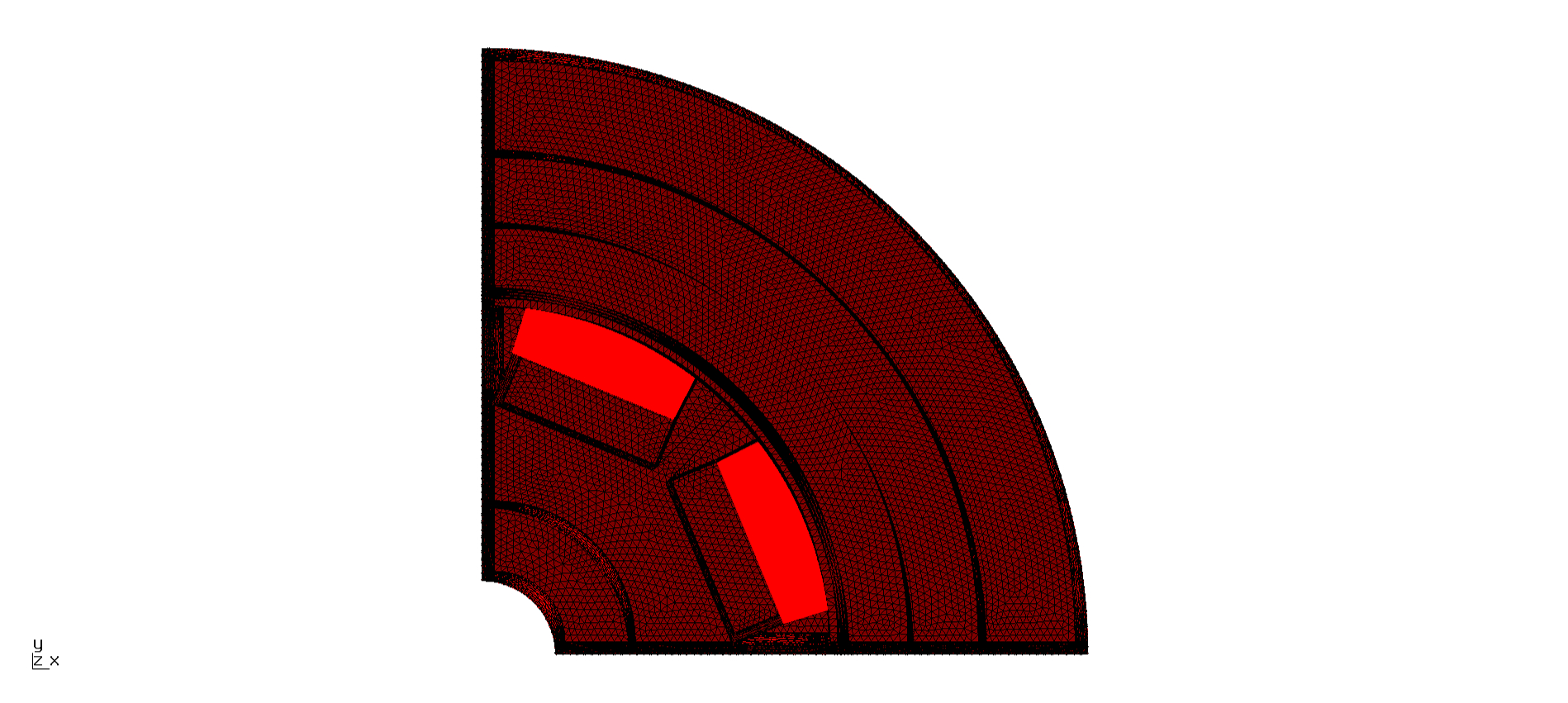} &
        \includegraphics[width=0.33\textwidth, trim=250 0 250 0, clip]{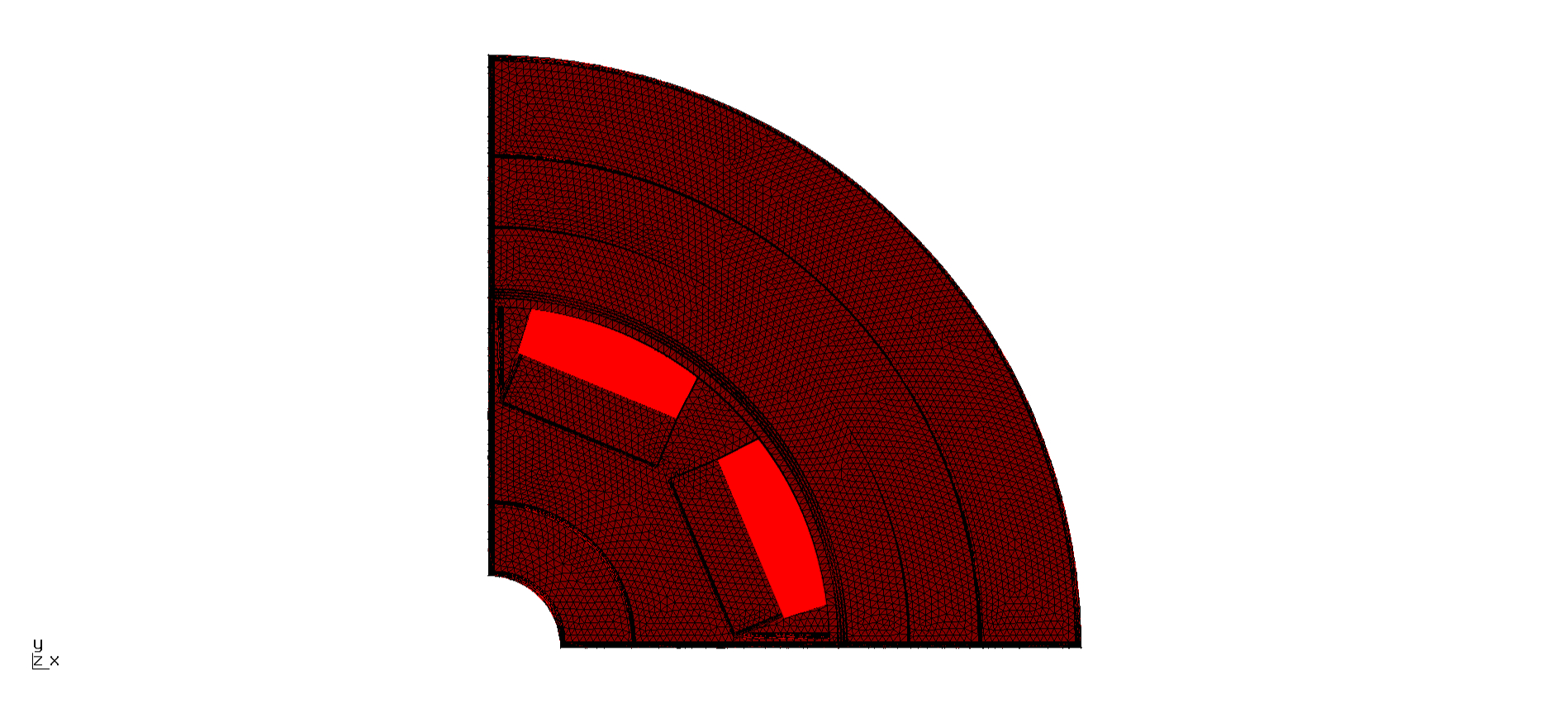} &
        \includegraphics[width=0.33\textwidth, trim=250 0 250 0, clip]{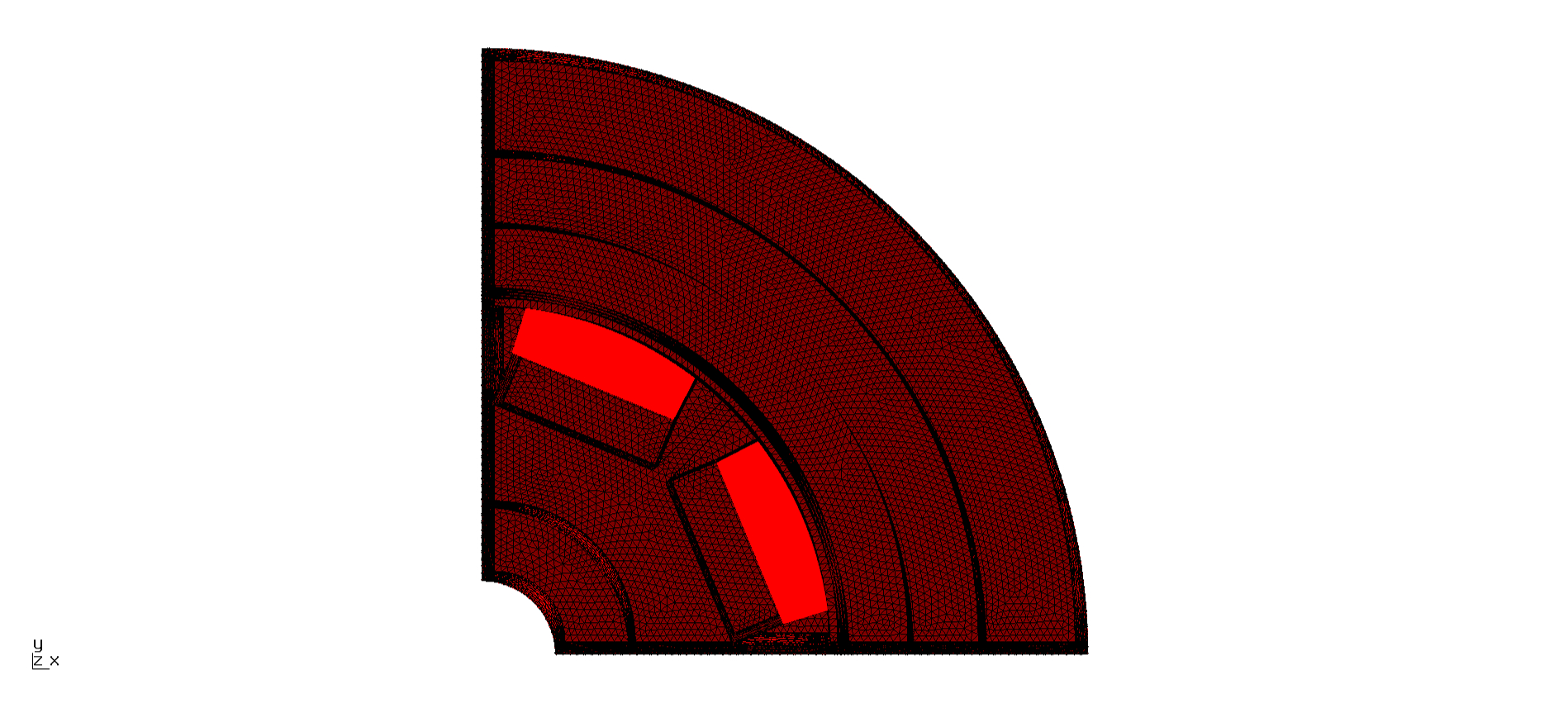} \\
        \includegraphics[width=0.33\textwidth, trim=250 0 250 0, clip]{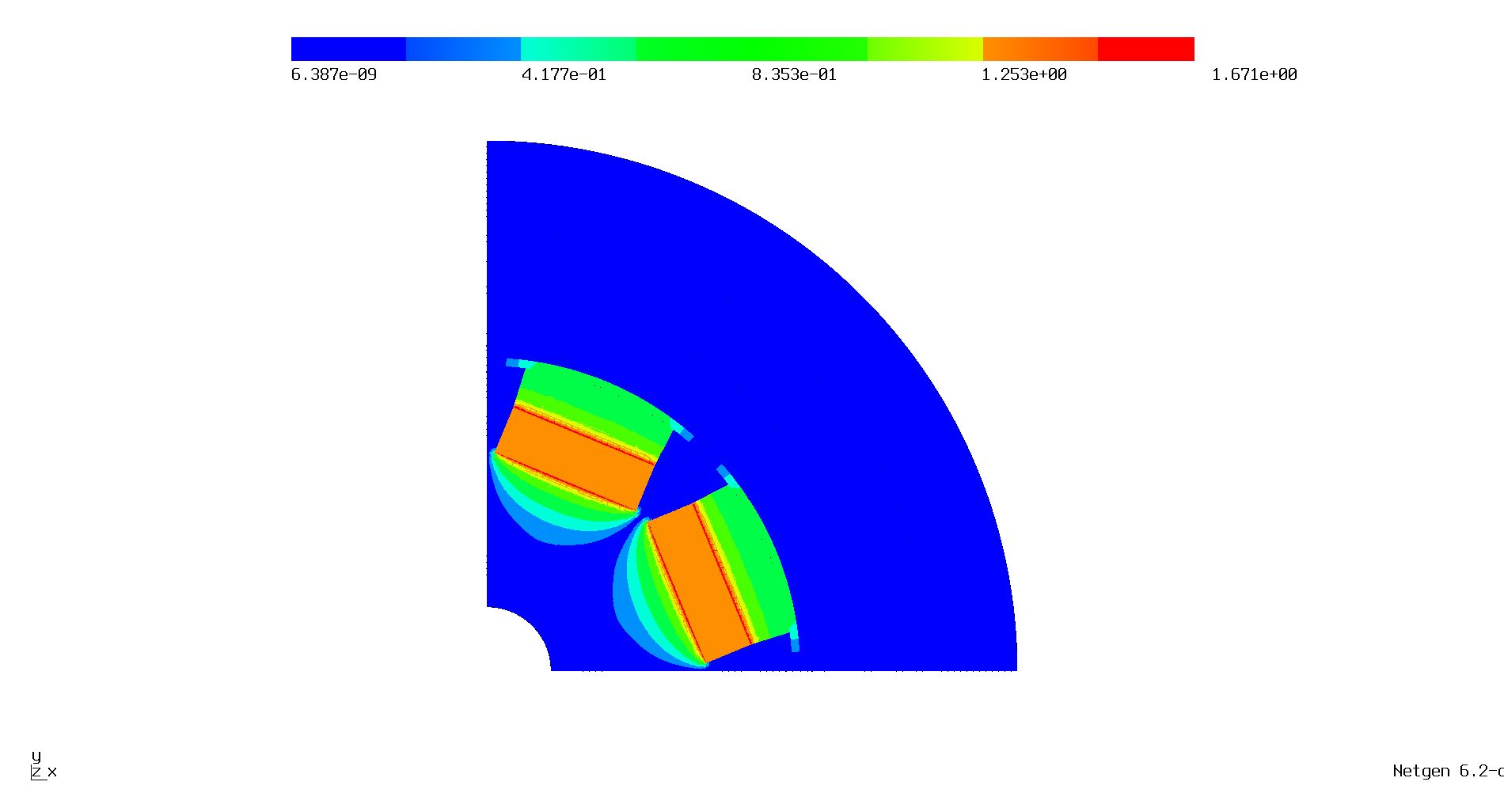} &
        \includegraphics[width=0.33\textwidth, trim=250 0 250 0, clip]{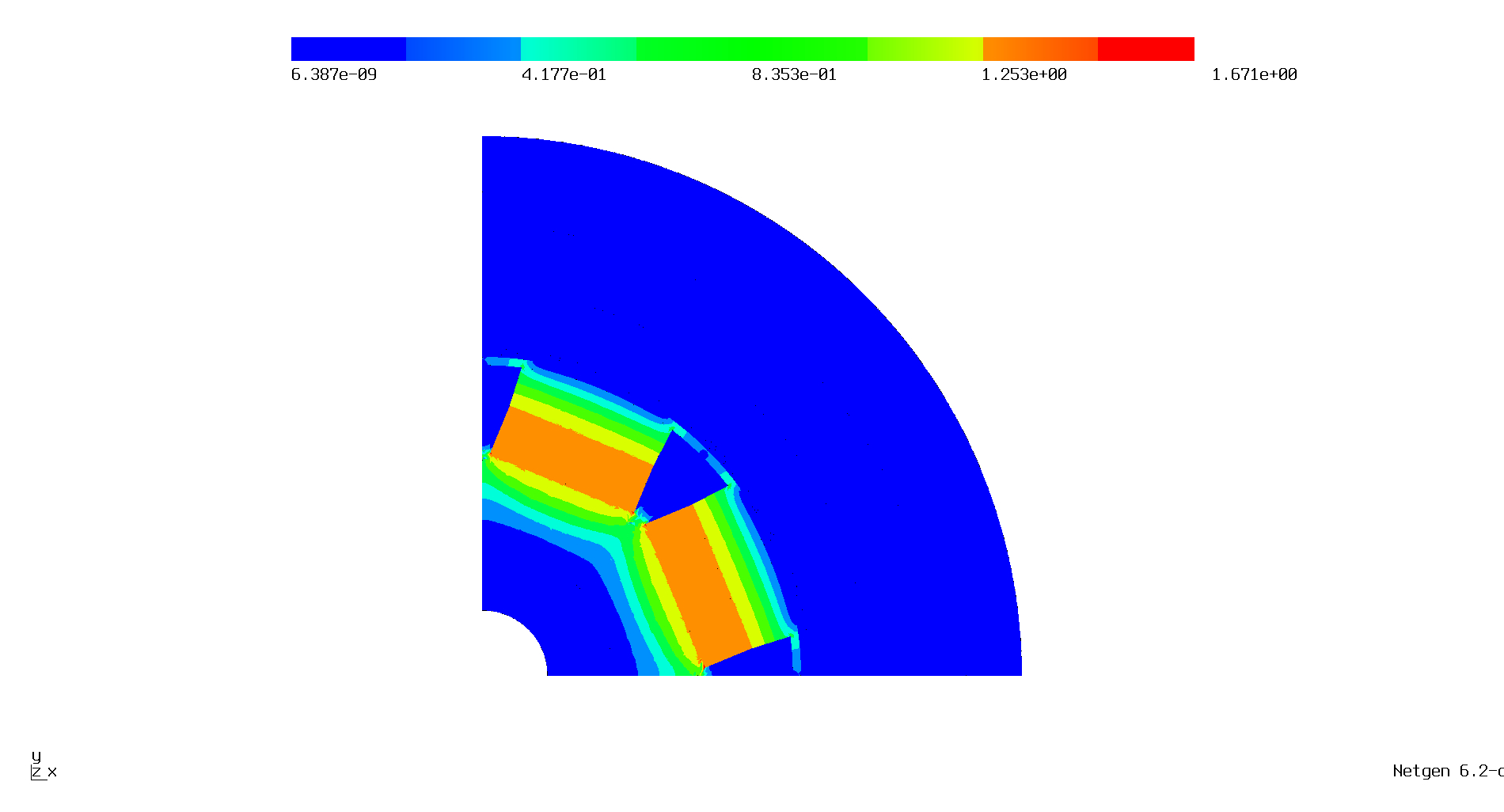} &
        \includegraphics[width=0.33\textwidth, trim=250 0 250 0, clip]{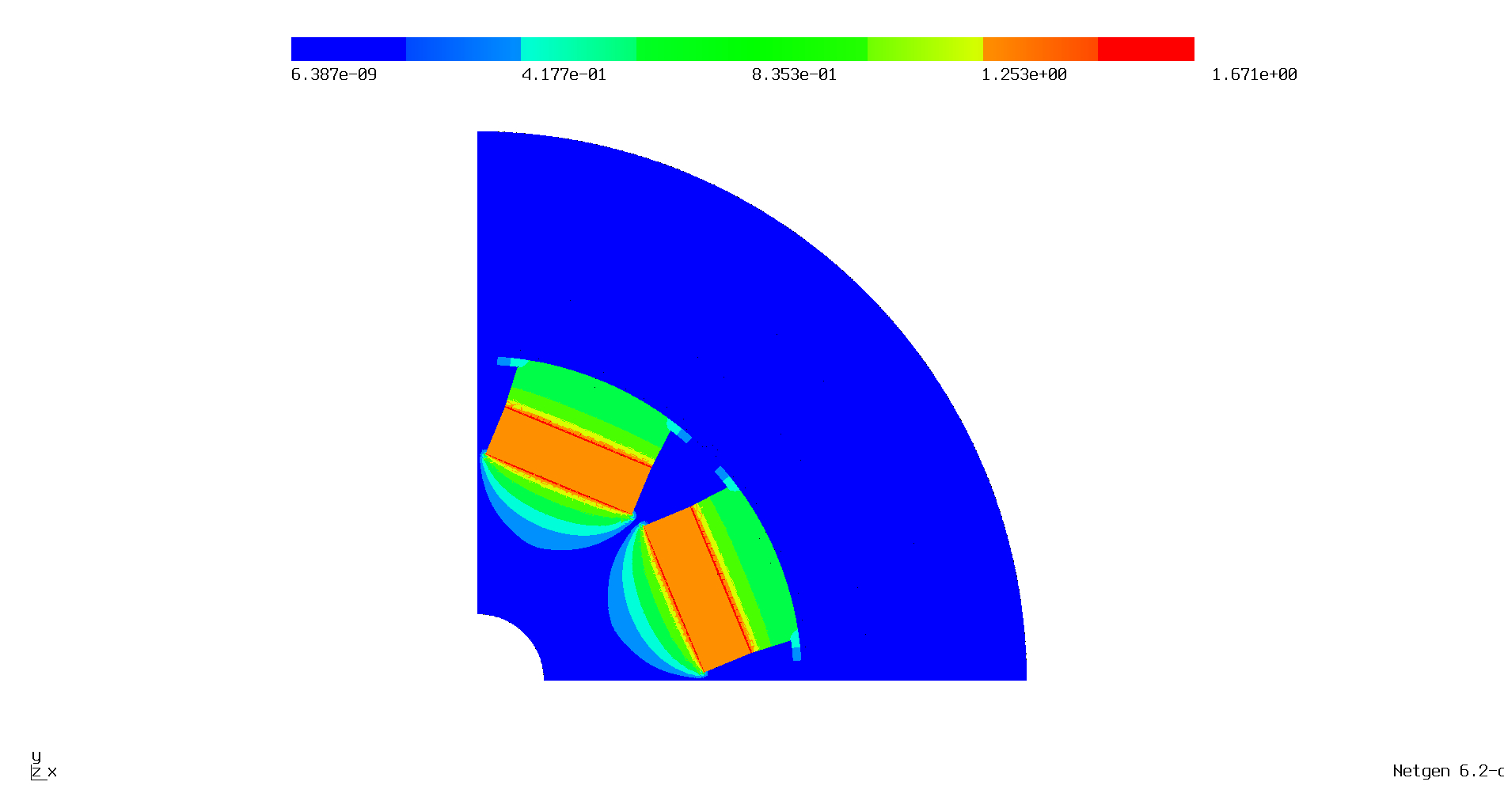} \\
        \includegraphics[width=0.33\textwidth, trim=250 0 190 0, clip]{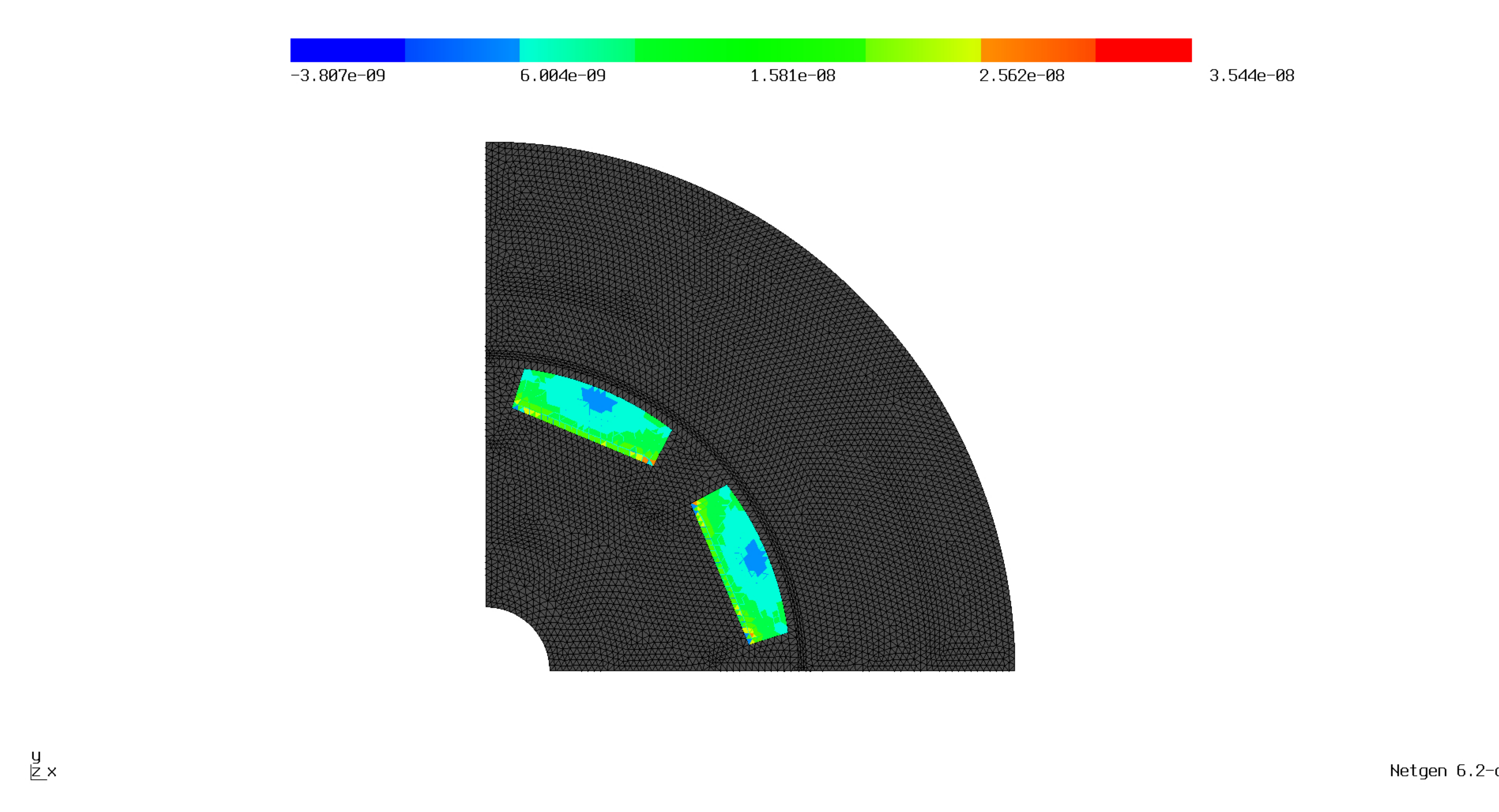} &
        \includegraphics[width=0.33\textwidth, trim=250 0 190 0, clip]{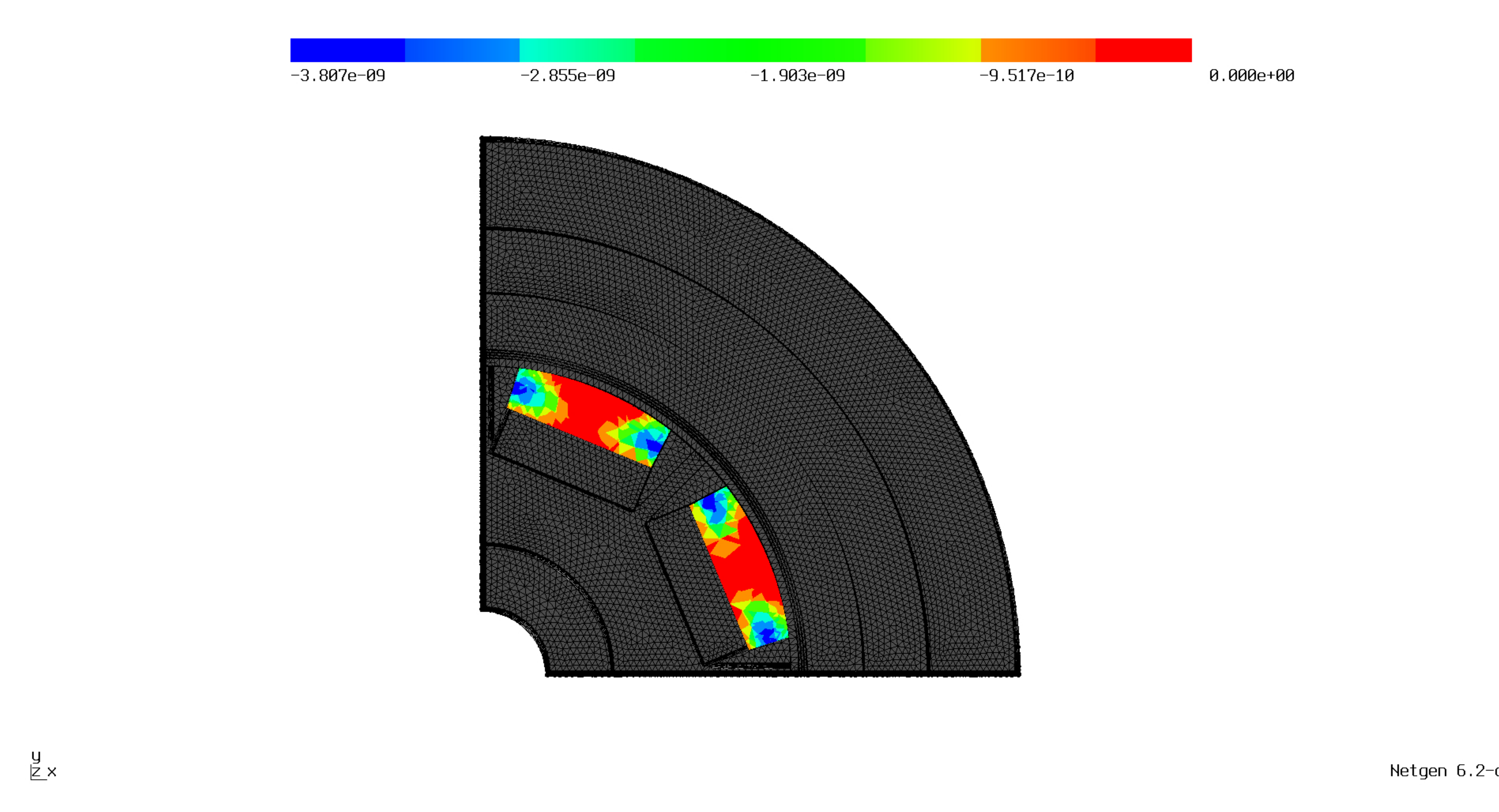} &
        \includegraphics[width=0.33\textwidth, trim=250 0 190 0, clip]{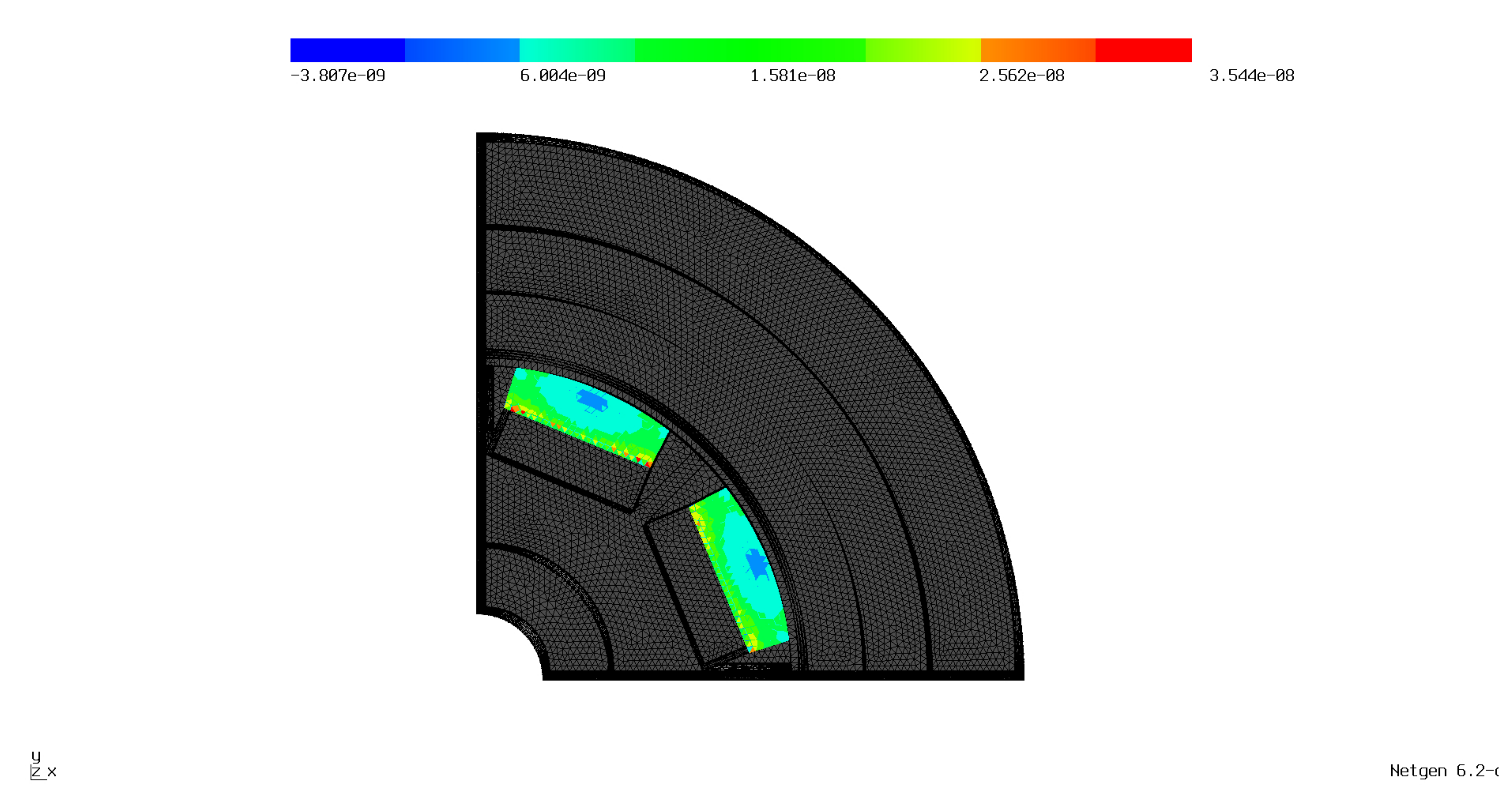} \\
    \end{tabular}
    \caption{Initial configuration, objective value $2.33 \cdot 10^{-8}$. 1st row: level set function. 2nd row: B-field ($\blue |B|\black  = \blue |\curl u |\black $). 3rd row: topological derivative. Left column: bottom. Central column: center. Right column: top.}
    \label{fig_init}
\end{figure}

\begin{figure}
    \begin{tabular}{ccc}
        \includegraphics[width=0.26\textwidth, trim=250 0 250 0, clip]{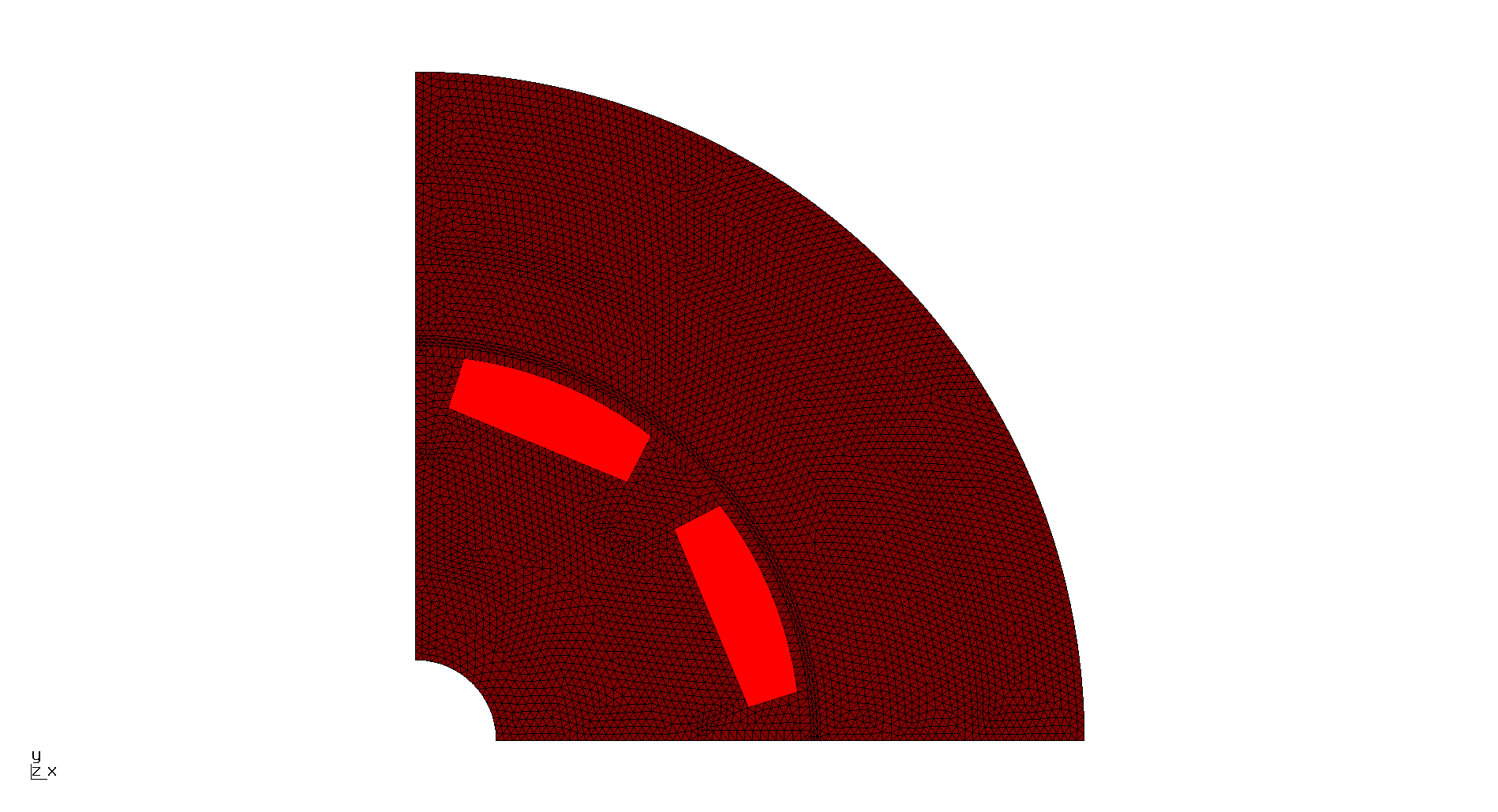} &
        \includegraphics[width=0.26\textwidth, trim=250 0 250 0, clip]{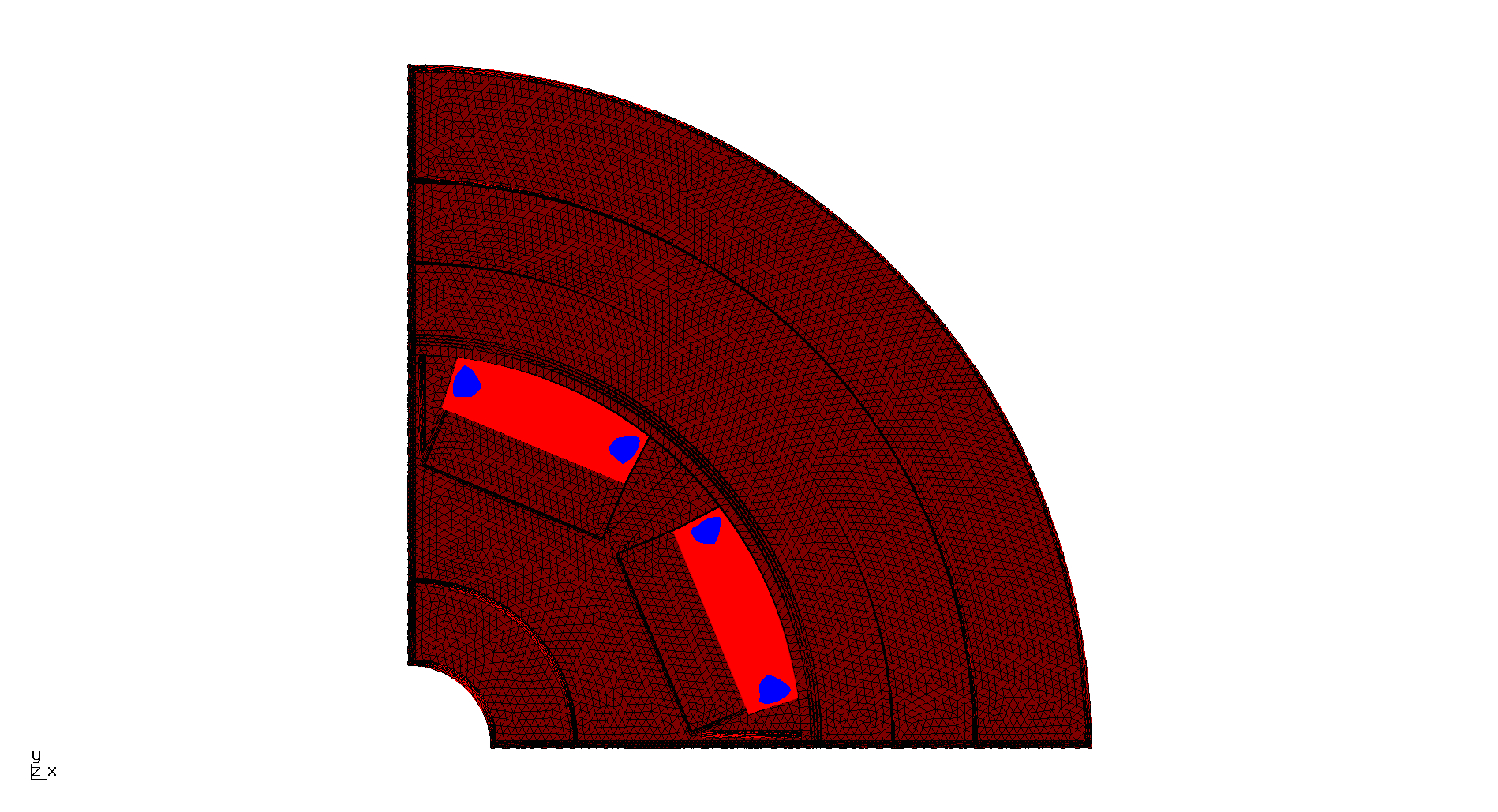}&
        \includegraphics[width=0.26\textwidth, trim=250 0 250 0, clip]{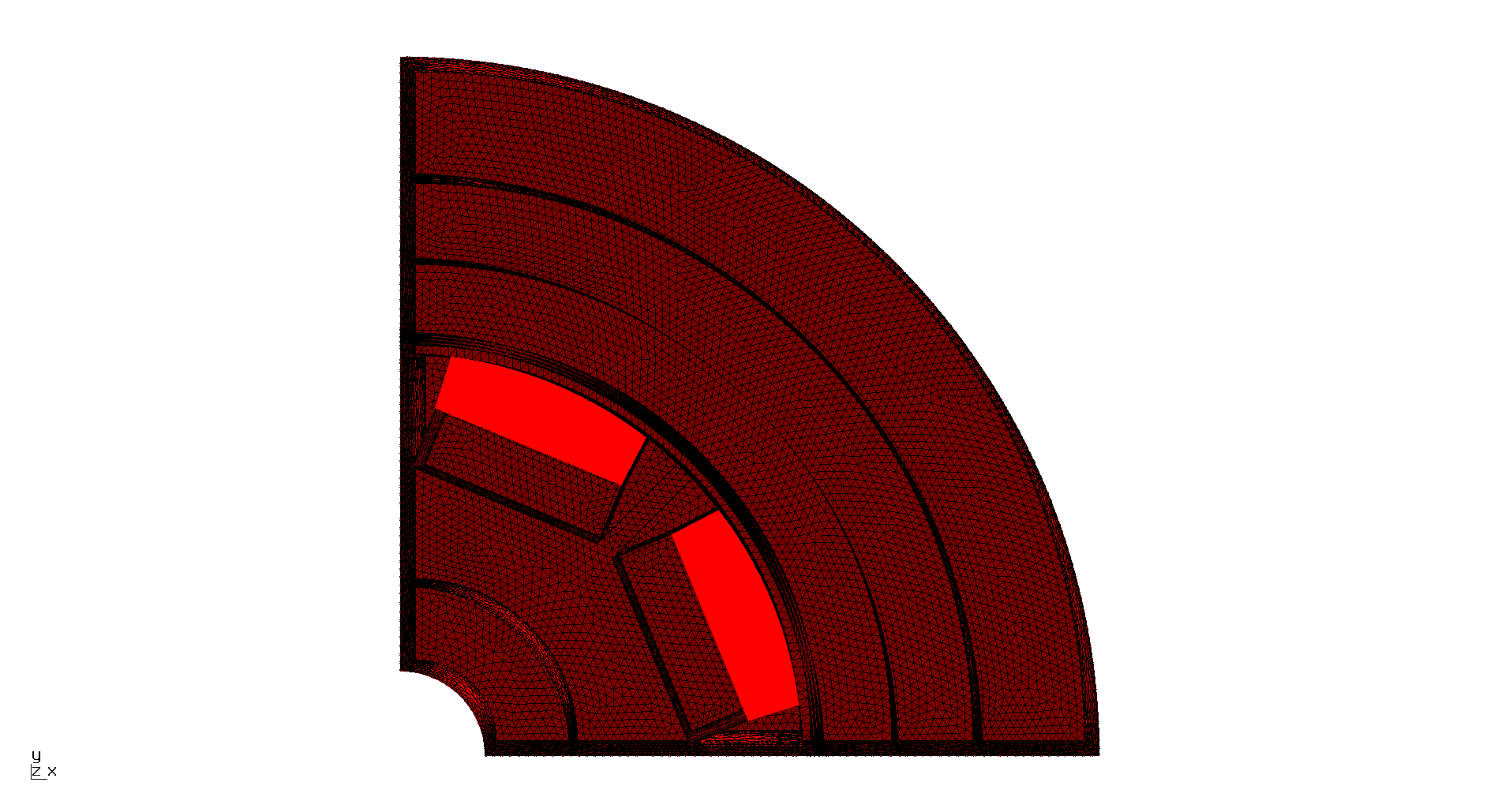} \\
        \includegraphics[width=0.33\textwidth, trim=250 0 250 0, clip]{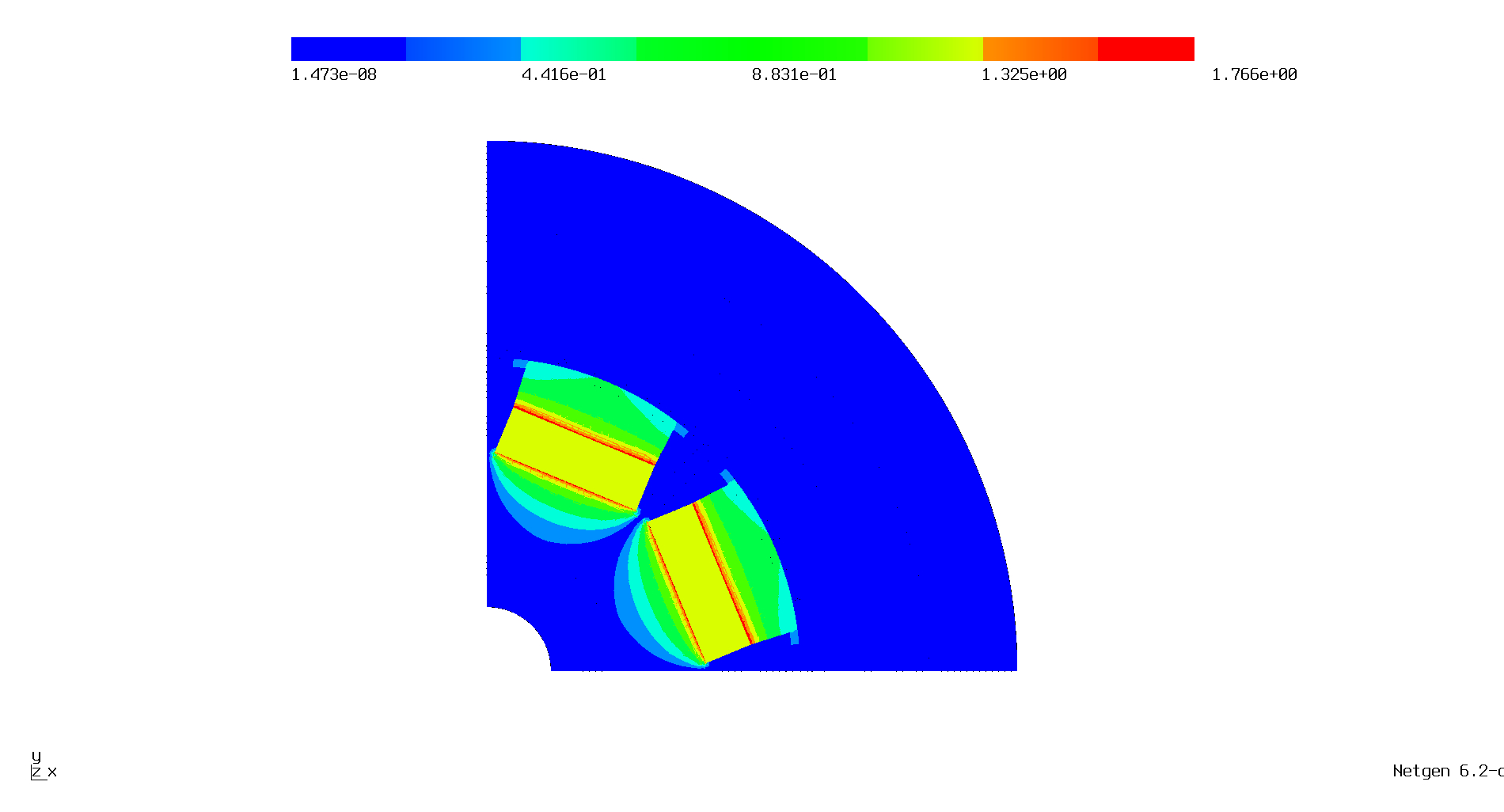} &
        \includegraphics[width=0.33\textwidth, trim=250 0 250 0, clip]{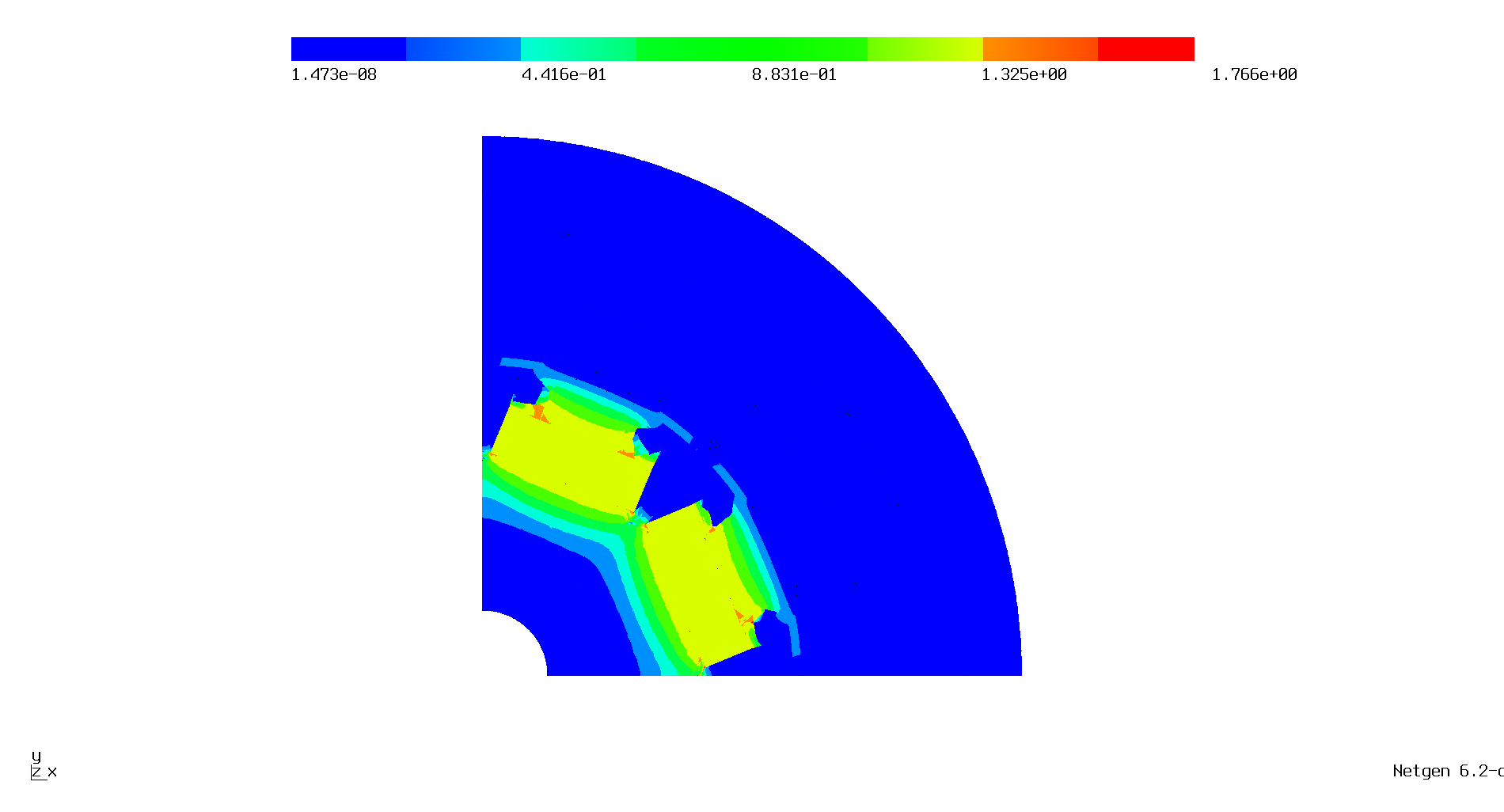} &
        \includegraphics[width=0.33\textwidth, trim=250 0 250 0, clip]{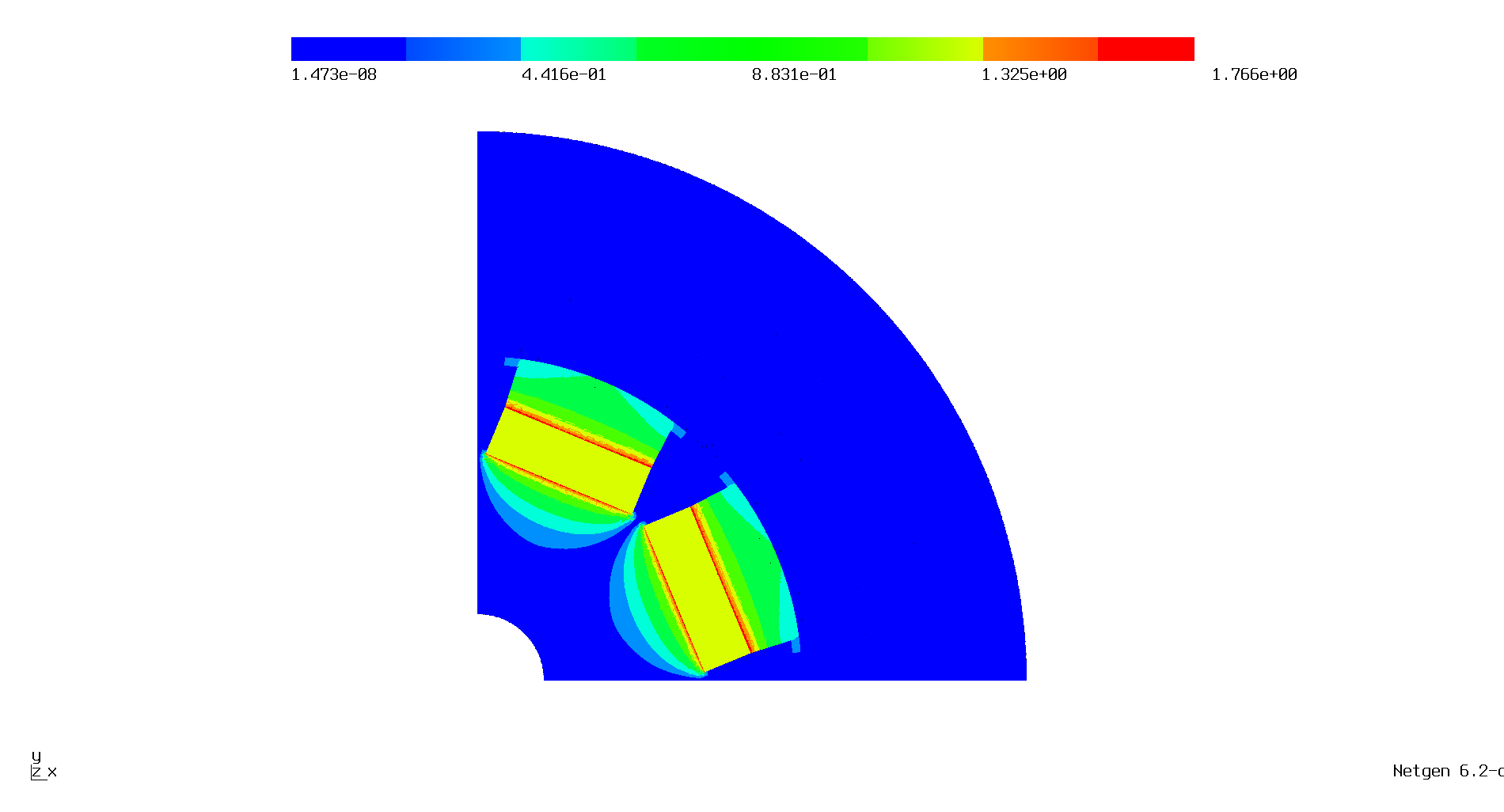} 
    \end{tabular}
    \caption{Improved configuration (after 1 iteration, objective value $4.68 \cdot 10^{-9}$). 1st row: level set function. 2nd row: B-field ($\blue |B| \black  = \blue |\curl u | \black $). 
    Left column: bottom. Central column: center. Right column: top.}
    \label{fig_iter1}
\end{figure}

\section*{Conclusion}
In this work we presented the rigorous derivation of the topological derivative for a class of quasi-linear $\curl$-$\curl$ problems under the assumption that $\curl u_0$ and $\curl p_0$ are (H\"older) continuous at the point where the topological perturbation takes place. We also discussed the efficient evaluation of the obtained formulas and applied our results to a physical model for an electrical machine. The results seem promising and show a significant improvement compared to the initial design.  

The magnetostatic model does not capture eddy currents. Therefore in a future work it would be interesting to consider the time-dependent magnetoquasistatic problem rather than the magnetostatics case. This however requires a thorough analysis and new tools have to be developed.

\section*{Appendix}
\subsection*{Lagrangian framework}\label{sec_adjFramework}
In this section we recall results on a Lagrangian framework. This section is taken from \cite[Sec. 2]{a_GanglSturm_2019a}.

\begin{definition}[parametrised Lagrangian]
 Let $X$ and $Y$ be vector spaces and $\tau >0$. A parametrised Lagrangian (or short Lagrangian) is a function
\begin{gather*}
(\eps,\fu,\fp) \mapsto  G(\eps,\fu,\fp): [0, \tau ] \times X \times Y \to \VR,
\end{gather*}
satisfying, 
        \ben
\fp\mapsto G(\eps,\fu,\fp) \quad \text{ is }  \text{affine} \text{ on } Y.
\een
\end{definition}

\begin{definition}[state and adjoint state]
Let $\eps \in [0,\tau]$.  We define the state equation by: find $u_\eps \in X$, such that 
  \ben\label{eq:state}
   \partial_\fp G(\eps,\fu_\eps,0)(\varphi)=0 \quad \text{ for all } \varphi\in Y. 
\een
The set of states is denoted $E(\eps)$. We define the adjoint state by: find $p_\eps\in Y$, such that 
\ben\label{eq:adjoint_state}
   \partial_u G(\eps,\fu_\eps,\fp_\eps)(\varphi)=0 \quad \text{ for all } \varphi\in  X. 
\een
The set of adjoint states associated with $(\eps,u_\eps)$ is denoted $Y(\eps,u_\eps)$.
\end{definition}

\begin{definition}[$\ell$-differentiable Lagrangian]\label{D:l_differentiable}
    Let $X$ and $Y$ be vector spaces and $\tau >0$. Let $\ell: [0,\tau] \to \VR$ be a given function satisfying $\ell(0)=0$ and $\ell(\eps)>0$ for 
$\eps \in (0,\tau]$. An $\ell$-differentiable parametrised Lagrangian is a parametrised Lagrangian $G:[0,\tau]\times X\times Y\to \VR$, satisfying,
\begin{itemize}
\item[(a)] for all $v,w\in X$ and $p\in Y$, 
    \ben\label{E:c1_lagrangian}
    s\mapsto G(\eps, v+s w, p) \text{ is continuously differentiable on } [0,1].
    \een
\item[(b)] for all $\fu_0\in E(0)$ and $\fp_0\in Y(0,\fu_0)$ the limit 
\ben
\partial_\ell G(0,\fu_0,\fp_0) := \lim_{\eps\searrow 0}\frac{G(\eps,\fu_0, \fp_0) - G(0,\fu_0, \fp_0)}{\ell(\eps)} \quad \text{ exists}.
\een
\end{itemize}
\end{definition}

\begin{assumption*}[H0]
    \begin{itemize}
     \item[(i)] We assume that for all $\eps \in [0,\tau]$, the set $E(\eps)=\{u_\eps\}$ is a singleton. 
     \item[(ii)] We assume that the adjoint equation for $\eps=0$, $\partial_u G(0,u_0,p_0)(\varphi)=0$ for all $\varphi \in E$, admits a unique solution.
    \end{itemize}
\end{assumption*}

We now give sufficient conditions when the function 
\ben\label{def_g}
\begin{split}
[0,\tau] &\to \VR \\
  \eps &\mapsto g(\eps) := G(\eps ,\fu_\eps,0),
  \end{split}
\een 
is one sided $\ell$-differentiable, that means, when the limit
\ben\label{E:dEll}
    d_{\ell}g(0):= \lim_{\eps\searrow 0}\frac{g(\eps)-g(0)}{\ell(\eps)}
    \een
    exists, where $\ell: [0,\tau] \to \VR$ is a given function satisfying $\ell(0)=0$ and $\ell(\eps)>0$ for $\eps \in (0,\tau]$. 
    
    \begin{theorem}[{\cite[Thm. 3.4]{a_GanglSturm_2019a} and \cite[Thm. 3.3]{Delfour2018}\;}] \label{thm:diff_lagrange}
    Let $G:[0,\tau]\times X\times Y \to \VR$ be an $\ell$-differentiable parametrised Lagrangian 
    satisfying Hypothesis~(H0). Define for $\eps > 0$, 
    \ben
    R_1^\eps (u_0,p_0):=   \frac{1}{\ell(\eps)}  \int_0^1 \left(\partial_u G(\eps,s u_\eps + (1-s)u_0, p_0) -   \partial_u G(\eps, u_0, p_0)\right)(u_\eps - u_0) \; ds
\een
and
\ben
\begin{split}
    R_2^\eps (u,p) :=\frac{1}{\ell(\eps)}(\partial_u G(\eps, u_0,p_0) - \partial_u G(0,u_0,p_0))(u_\eps - u_0).
\end{split}
\een
If $R_1(u_0,p_0) := \lim_{\eps\searrow 0} R^\eps_1(u_0,p_0)$ and $R_2(u_0,p_0) := \lim_{\eps\searrow 0} R^\eps_2(u_0,p_0)$ exist, then  
\benn
d_\ell g(0) = \partial_\ell G(0,\fu_0,\fp_0) + R_1(u_0,p_0) + R_2(u_0,p_0).
\eenn
\end{theorem}

\bibliography{topological_3D_rev1}
\bibliographystyle{plain}
\end{document}

%% file: slidesymbols.tex
\providecommand{\Div}{\operatorname{div}}          
\providecommand{\curl}{\operatorname{{ curl}}}  

\newcommand{\VR}{{\mathbf{R}}}

\newcommand{\Dsf}{\mathsf{D}}

\providecommand{\Ca}{{\cal A}}